\documentclass[12pt, oneside,reqno]{amsart}   	
\usepackage{geometry}
\usepackage{float}
\allowdisplaybreaks
\usepackage{graphicx}				
\usepackage{amssymb}
\usepackage{amsmath}
\usepackage[all]{xy}
\usepackage{mathtools}
\usepackage{tikz-cd}
\usepackage{enumitem}
\usepackage{tikz-cd}
\usepackage{appendix}

\usepackage[%
bookmarks=true,
colorlinks,
linkcolor=blue,
urlcolor=blue,
citecolor=blue,
plainpages=false,
pdfpagelabels,
final,
breaklinks=true\between
]{hyperref}
\usepackage{hyperref}
\usepackage[numbers,sort&compress]{natbib}

\newtheorem{thm}{Theorem}[section]
\newtheorem{lem}[thm]{Lemma}

\newtheorem{rem}{Remark}[section]

\newtheorem{defn}{Definition}[section]

\numberwithin{equation}{section}

\newcommand{\Balpha}{\mbox{$\hspace{0.12em}\shortmid\hspace{-0.62em}\alpha$}} 

\def\Pb{\ifmmode{\Bbb P}\else{$\Bbb P$}\fi}
\def\Z{\ifmmode{\Bbb Z}\else{$\Bbb Z$}\fi}
\def\C{\ifmmode{\Bbb C}\else{$\Bbb C$}\fi}
\def\R{\ifmmode{\Bbb R}\else{$\Bbb R$}\fi}
\def\S{\ifmmode{S^2}\else{$S^2$}\fi}

\def\S{\cal S}

\newenvironment{pf}{\paragraph{Proof:}}{\hfill$\square$ \newline}

\begin{document}

\title[unknottedness]{An unknottedness result for noncompact self shrinkers}
\author{Alexander Mramor}
\address{Department of Mathematics, Johns Hopkins University, Baltimore, MD, 21231}
\email{amramor1@jhu.edu}

\begin{abstract} In this article we extend an unknottedness theorem for compact self shrinkers to the mean curvature flow to shrinkers with finite topology and one asymptotically conical end, which conjecturally comprises the entire set of self shrinkers with finite topology and one end. The mean curvature flow itself is used in the argument presented.
\end{abstract}
\maketitle
\section{Introduction}

Self shrinkers are the most basic singularity models to the mean curvature flow and hence are an important topic of study. In this article we extend (and reprove) the results of \cite{MW1}, where the author with S. Wang showed compact self shrinkers in $\R^3$ are topologically standard, to include some noncompact self shrinkers:

\begin{thm}\label{thm:unknotted} Let $M^2 \subset \R^3$ be a two-sided embedded self shrinker which is either compact or has a single, asymptotically conical, end. Then $M$ is topologically standard. 

\end{thm} 

The definition of standard embeddedness is given in section 4 below but  in layman's terms it essentially means that a surface is ``unknotted'': for example tubular neighborhoods of knotted $S^1 \subset \R^3$ are not topologically standard. From theorem 2 of Brendle's paper \cite{Bren} noncompact shrinkers where any two curves have vanishing mod 2 intersection number must be the cylinder or plane and hence unknotted; otherwise to the author's knowledge no other unknottedness results for noncompact self shrinkers prior to this one had been shown without some extra symmetry or curvature convexity assumptions on the shrinker, like mean convexity. 
$\medskip$

From the desingularization of the sphere and plane by Kapouleas, Kleene, and M{\o}ller \cite{KKM} we see that the set of asymptotically conical shrinkers with one end and finite topology contains elements with nontrivial topology. In \cite{I2} Ilmanen conjectured that a self shrinker with a cylindrical end must itself be the round shrinking cylinder, so by L. Wang's theorem \cite{Lu} on the ends of noncompact self shrinkers with finite topology in the conjectural picture our result covers all self shrinkers with one end and finite topology. Note another of result of theirs \cite{Lu1}, giving rigidity for when a shrinker is very quickly asymptotic to a cylinder, provides some concrete evidence supporting Ilmanen's conjecture. 
$\medskip$


The germ of the main topological argument in this article, employing Waldhausen's theorem, can be found in the important paper of Lawson \cite{L} for minimal surfaces in $S^3$. There as discussed more below a criterion which implies unknottedness of minimal embeddings is given which can be verified using Frankel's theorem, that two minimal surfaces in a space of positive Ricci curvature must intersect. Frankel theorems for self shrinkers also hold, but for technical reasons an argument appealing directly to them seem to encounter some difficulties at least until the end of our proof. These issues and related literature and techniques in the classical minimal surface case are discussed in some depth in the concluding remarks below.
$\medskip$

 To work around these issues we will use the (renormalized) mean curvature flow. In a nutshell, its use will be that if a shrinker satisfying the conditions of theorem \ref{thm:unknotted} is in fact topologically nonstandard, we may perturb it appropriately and use the flow described in the statement below. The shrinker being nonstandard gives by an intersection number argument in the universal cover of one of its bounded components that the flow is nonempty and in particular intersects some bounded domain of $\R^3$ for all times. By deep work of White, in the limit as $t \to \infty$ we then find another shrinker which must have impossible properties for multiple reasons: stable in the Gaussian metric but polynomial volume growth, additionally one which is disjoint from the original shrinker. A potential alternate route along the same lines instead using Meeks, Simon, Yau \cite{MSY} to find an ``impossible shrinker'' as in the sketch above, by minimization in isotopy class, is also discussed in the concluding remarks. The construction of the flow we use to carry this scheme out is the following shown in section 3: 
\begin{thm}\label{LSF} Let $M \subset \R^3$ be an asymptotically conical surface such that $H - \frac{\langle X, \nu \rangle}{2} \geq c(1 + |X|^2)^{-\alpha}$ for some constants $c, \alpha >0$ and choice of normal, and so that $|A(q)|^2 \to 0$ as $R \to \infty$ for any $q \in M \cap B(p, R)^c$. Then denoting by $K$ the region bounded by $M$ whose outward normal corresponds to the choice of normal on $M$, the level set flow $M_t$ of $M$ with respect to the renormalized mean curvature flow satisfies
\begin{enumerate}
\item inward in that $K_{t_1} \subset K_{t_2}$ for any $t_1 > t_2$, considering the corresponding motion of $K$.
\item $M_t$ is the Hausdorff limit of surgery flows $S_t^k$ with initial data $M$. 
\item $M_t$ is a forced Brakke flow (with forcing term given by position vector). 
\end{enumerate}
\end{thm} 
The definitions of weak flows needed are given as they come up in section 3. The level set flow for compact sets under the renormalized flow (and with more general forcing terms) has been well studied, see for example the work of Hershkovits and White \cite{HW1} and also \cite{HW} by the same authors for a use of the renormalized flow in studying the homotopy groups of self shrinkers. The renormalized mean curvature flow on noncompact hypersurfaces seems to be less well studied/used in situations where singularity formation is not ruled out; here we do so in a rather particular context where only moderate adjustements are needed to adapt the techniques from the compact case. Most relevent is the work of Choi, Chodosh, Mantoulidis, Schulze \cite{CCMS} in their construction of a generic mean curvature flow and in fact below we adapt some of their arguments to our situation, particularly in showing item (3) above. The explicit realization of the level set flow as a limit of surgery flows, which was noted in the compact case by Lauer and Head \cite{Lau, Head1} is useful to show the flow is nonempty in its ultimate application to the proof of theorem \ref{thm:unknotted}.
$\medskip$

$\textbf{Acknowledgements:}$ The author is grateful to Bill Minicozzi for asking him about unknottedness for noncompact shrinkers at the 2018 Barrett lectures held at the University of Tennessee--Knoxville, which inspired this work. He also thanks Jacob Bernstein, Letian Chen, Martin Lesourd, Peter McGrath, Joel Spruck, Ao Sun, Ryan Unger, Shengwen Wang, and Jonathan Zhu for stimulating discussions and feedback during various stages of this project. The author also thanks the anonymous referees for their thoughtful reviews whose questions and critique helped improve the article.

\section{A brief introduction the mean curvature flow}

In this section we discuss some facts concerning the mean curvature flow and self shrinkers which we will use below, though we elect to postpone some other background material to other sections where they seem to fit more naturally into the discussion. Let $X_0: M \to N^{n+1}$ be an embedding of $M$ realizing it as a smooth closed hypersurface of $N$, which by abuse of notation we also refer to as $M$. Then the mean curvature flow of $M_t$ is given by the image of $X: M \times [0,T) \to N^{n+1}$ satisfying the following where $\nu$ is the outward normal: 
\begin{equation}\label{MCF equation}
\frac{dX}{dt} = \vec{H} = -H \nu, \text{ } X(M, 0) = X_0(M)
\end{equation}
 This is an interesting flow to consider for a variety of reasons, for example as a tool in topology -- for a survey see \cite{survey}. There is a comparison principle for the mean cuvature flow which says that any two flows where one is compact must stay disjoint if they are initially, and using it by enclosing any compact hypersurface in $\R^{n+1}$ with a sufficiently large sphere it is easy to see that singularities are common for mean curvature flows in Euclidean space. Generically the only noncompact singularities encountered will be modeled on round cylinders: generic mean curvature flow in $\R^3$ has been already rather well developed (see \cite{CM} and \cite{CCMS}) although there still might be situations where one is forced to consider ``exotic'' singularities, for example in potential applications of the flow to a family of surfaces considered simultaneously. 
$\medskip$

 To study these singularities, one may perform a $\textit{tangent flow blowup}$ which, as described by Ilmanen in his preprint \cite{I} for flows of surfaces in $\R^3$, will be modeled on smooth self shrinkers possibly with some mutiplicity: these are hypersurfaces satisfying the following equivalent definitions: 
 \begin{enumerate}
\item $M^n \subset \R^{n+1}$ which satisfy $H - \frac{\langle X, \nu \rangle}{2} = 0$, where $X$ is the position vector.
\item  Minimal hypersurfaces in the Gaussian metric $G_{ij} = e^{\frac{-|x|^2}{2n}} \delta_{ij}$.
\item Hypersurfaces $M$ which give rise to ancient flows $M_t$ that move by dilations by setting $M_t = \sqrt{-t} M$, $t < 0$. 
\end{enumerate}

Of course, as the degenerate neckpinch of Angenent and Velasquez \cite{AV} illustrate tangent flows do not capture quite all the information about a developing singularity but they are a natural starting point and a lot can be said about the singular set by considering them. Taking the viewpoint of them as minimal surfaces its important to note the Gaussian metric is poorly behaved in many regards; it is incomplete and by the calculations in \cite{CM1} its scalar curvature at a point $x$ is given by:
\begin{equation}
R = e^{\frac{|x|^2}{2n}}\left( n+ 1 - \frac{n-1}{4n} |x|^2 \right)
\end{equation}
We see that as $|x| \to \infty$ the scalar curvature diverges, so there is no way to complete the metric. Also since $R$ is positive for $|x|$ small and negative for large $|x|$, there is no sign on sectional or Ricci curvatures. On the other hand it is $f$-Ricci positive, in the sense of Bakry and Emery with $f = -\frac{1}{2n} |x|^2$, suggesting it should satisfy many of the same properties of true Ricci positive metrics (see \cite{WW}). Indeed, this provides some idea as to why one might expect an unknottedness result for self shrinkers, because analogous unknottedness results hold in Ricci positive metrics on $S^3$ as discussed later.
$\medskip$

By this analogy as one might expect it turns out that there are no stable minimal surfaces of polynomial volume growth in $\R^n$ endowed with the Gaussian metric as discussed in \cite{CM1}. To see why this is so, the Jacobi operator for the Gaussian metric is given by: 
\begin{equation}
L = \Delta + |A|^2 - \frac{1}{2} \langle X, \nabla(\cdot) \rangle + \frac{1}{2}
\end{equation} 
The extra $\frac{1}{2}$ term is essentially the reason such self shrinkers unstable in the Gaussian metric: for example owing to the constant term its clear in the compact case from this that one could simply plug in the function ``1'' to get a variation with $Lu >0$ which doesn't change sign implying the first eigenvalue is negative. 
$\medskip$

To deal with this instability, in \cite{CM} Colding and Minicozzi introduced their entropy functional which is essentially an area that mods out by translations and dilations. They define the entropy $\lambda(M)$ of $M^n \subset \R^{n+1}$ to be:
\begin{equation}
\lambda(M) = \sup\limits_{x_0, r} F_{x_0, r}(M)
\end{equation}
where the functionals $F_{x_0, r}$ are Gaussian areas shifted by $x_0$ and rescaled by $r$ -- although it doesn't concern us there are indeed entropy stable shrinkers namely round spheres and cylinders. What does concern us is that the entropy by Huisken monotonicity \cite{H} is nonincreasing under the flow and as shown lemma 2.9 in \cite{CM} a surface with finite entropy has polynomial volume growth. And in fact, every properly embedded shrinker has polynomial volume growth by Q. Ding and Y.L. Xin: 
\begin{thm}[Theorem 1.1 of \cite{DX}]\label{proper} Any complete non-compact properly immersed self-shrinker $M^n$ in $\R^{n+m}$ has Euclidean volume growth at most. \end{thm}
We will combine these facts below to conclude the self shrinker we find via the renormalized flow is unstable in the Gaussian metric. Now we discuss some terminology describing possible behavior of the ends: 
$\medskip$

A \emph{regular cone} in $\Bbb R^3$ is a surface of the form $C_\gamma=\{r\gamma\}_{r\in (0,\infty)}$ where $\gamma$ is smooth simple closed curve in $S^2$. An end of a surface $M^2\hookrightarrow \R^3$ is \emph{asymptotically conical} with asymptotic cross section $\gamma$ if $\rho M\to C_\gamma$ in the $C^2_\mathrm{loc}$ sense of graphs as $\rho\searrow 0$ restricted to that end. 
$\medskip$

Similarly we define \emph{asymptotically cylindrical} ends to be ends which are asymptotically graphs over cylinders (with some precsribed axis and diameter) which converge to that cylinder in $C^2_{loc}$ on that end. The reason we focus on such ends is the following important result of L. Wang, which says that these are the only possible types of ends which may arise in the case of finite topology: 
\begin{thm}[Theorem 1.1 of \cite{Lu}]\label{Lu-ends}

If M is an end of a noncompact self-shrinker in $\R^3$ of finite topology, then either of the following holds:
\begin{enumerate}
\item $\lim_{\tau \to \infty} \tau^{-1} M = C(M)$ in $C_{loc}^\infty(\R^3 \setminus 0)$ for $C(M)$ a regular cone in $\R^3$
\item $\lim_{\tau \to \infty} \tau^{-1} (M - \tau v(M)) = \R_{v(M)} \times S^1$ in $C_{loc}^\infty(\R^3)$ for a $v(M) \in \R^3 \setminus \{0\}$
\end{enumerate}
\end{thm}
In particular, theorem \ref{Lu-ends} applies to self shrinkers which arises as the tangent flow to compact mean curvature flows, although it is true one should expect shrinkers with more than one end to appear in a general blowup (for a trivial example consider a neckpinch). We end this discussion with a pseudolocality theorem. Pseudolocality roughly says that far away points are less consequential under the flow than nearby ones no matter their curvature and is a concrete artifact of the nonlinearity of the flow. In our case it is a consequence of the Ecker-Huisken estimates \cite{EH} but we give the formulation of B.L. Chen and L. Yin (see theorem 1.4 in \cite{INS} for a proof in $\R^n$ by controlling Gaussian densities). It will be heavily used in the extension of the flow with surgery below:

\begin{thm}[Theorem 7.5 of \cite{CY}]\label{PL} Let $\overline{M}$ be an $\overline{n}$-dimensional manifold satisfying $\sum\limits_{i=0}^3 | \overline{\nabla}^i \overline{R}m| \leq c_0^2$ and inj$(\overline{M}) \geq i_0 > 0$. Then there is $\epsilon >0$ with the following property. Suppose we have a smooth solution $M_t \subset \overline{M}$ to the MCF properly embedded in $B_{\overline{M}}(x_0, r_0)$ for $t \in [0, T]$ where $r_0 < i_0/2$, $0 < T \leq \epsilon^2 r_0^2$. We assume that at time zero, $x_0 \in M_0$, and the second fundamental form satisfies $|A|(x) \leq r_0^{-1}$ on $M_0 \cap B_{\overline{M}}(x_0, r_0)$ and assume $M_0$ is graphic in the ball $B_{\overline{M}}(x_0, r_0)$. Then we have 
\begin{equation}
|A|(x,t) \leq (\epsilon r_0)^{-1}
\end{equation}
for any $x \in B_{\overline{M}}(x_0, \epsilon r_0) \cap M_t$, $t \in [0,T]$. 
\end{thm}

\section{The renormalized mean curvature flow}
In this section we discuss the renormalized mean curvature flow (which we'll abreviate RMCF) ultimately to construct, via an adapted surgery flow, an inward level set flow for the RMCF so showing theorem \ref{LSF}. Using the same notation as in the section above for a surface $M \subset \R^3$ the RMCF is given by:

\begin{equation}\label{renorm1}
\frac{dX}{dt} =  \vec{H} + \frac{X}{2}
\end{equation}
Modding out by tangential directions of the flow makes the speed of the flow more transparent and is geometrically equivalent to \ref{renorm1}: 
\begin{equation}\label{renorm}
\frac{dX}{dt} = -(H - \frac{\langle X, \nu \rangle}{2}) \nu
\end{equation}

Where here as before $X$ is the position vector on $M$. It is related to the regular mean curvature flow by the following reparameterization; this will allow us to transfer many deep analytical properties of the MCF to the RMCF. Supposing that $M_t$ is a mean curvature flow on $[-1,T)$, $-1 < T \leq 0$ then the renormalized flow $\hat{M_\tau}$ of $M_t$ defined on $[0, -\log(-T))$ is given in terms of position vector and time by:
\begin{equation}\label{param}
\hat{X}_\tau = e^{\tau/2} X_{-e^{-\tau}},\text{ } \tau = -\log{(-t)}
\end{equation}

One can check this is in agreement with equation \ref{renorm1}. This is a natural flow for us to consider because it is up to a multiplicative term the gradient flow of the Gaussian area and fixed points with respect to it are precisely self shrinkers. More precisely, writing $H_G$ for the mean curvature of a surface with respect to the Gaussian metric one finds from the standard formula for the mean curvature under conformal change that:
\begin{equation}\label{Hrelation}
H_G = e^{\frac{|x|^2}{8} }(H - \frac{\langle X, \nu \rangle}{2})
\end{equation}
$\medskip$

It's clear from this that the RMCF should be better behaved then the MCF in the Gaussian metric then because of the missing exponential factor in the speed of the flow; in fact the surfaces we consider in the sequel will be well behaved with respect to the RMCF but will have unbounded mean curvature in the Gaussian metric. Important for below, since $t = -e^{-\tau}$, $H = e^{\tau/2} \hat{H}$, $X= e^{-\tau/2}\hat{X}$ and the unit normal is unchanged we have: 
\begin{equation}\label{change}
-tH - \frac{\langle X, \nu \rangle}{2} \to e^{-\tau/2}(\hat{H} - \frac{\langle \hat{X}, \nu \rangle}{2})
\end{equation}
Note that throughout when we refer only to the RMCF we will use the notation typical to the MCF (i.e. $t$ instead of $\tau$, etc.). We will make the distinction clear when we refer to the ``regular'' MCF. 
$\medskip$

 Our main object of study in this section will be the following set, which we bold for emphasis; the asymptotics assumed are inspired by Bernstein and Wang \cite{BW} for use with Ecker and Huisken's noncompact maximum principle in \cite{EH} as we'll use shortly:

\begin{defn} Denote by $\Sigma$ the set of asymptotically conical hypersurfaces in $\R^3$ for which $H - \frac{\langle X, \nu \rangle}{2} \geq c(1 + |X^2|)^{-\alpha}$ for some constants $c, \alpha > 0$ and choice of outward normal. 
\end{defn}

Throughout, we'll say that $M$ is $\textit{shrinker mean convex}$ if $H - \frac{\langle X, \nu \rangle}{2} \geq 0$ at all points on $M$. First we note that short time existence of the RMCF of these surfaces:

\begin{lem}  If $M \in \Sigma$ then there exists some $\epsilon > 0$ for which the classical RMCF $M_t$ of $M$ exists for $t \in [0, \epsilon)$. 
\end{lem} 
\begin{pf} 
We can flow an element in $\Sigma$ by the regular MCF for a short time by Ecker-Huisken \cite{EH}; then apply the reparameterization \ref{param} to get a solution for short times for the RMCF. 
\end{pf}

Our next lemma is that shrinker mean convexity is preserved under the RMCF; in our future application to the flow with surgery, note that this must be reapplied (resetting $t = 0$, adjusting the constant $c$ appropriately) after every surgery time since high curvature regions will be removed: 
\begin{lem}\label{con} Let $M_t$ be a smooth flow under RMCF on $[0,T]$. Then if it is initially in $\Sigma$ it remains so under the MCF and in fact: 
\begin{equation}\label{lb}
(H - \frac{\langle X, \nu \rangle}{2}) > ce^{t/2}(1 + e^{-t}|X|^2 + 2n(-e^{-t}+1))^{-\alpha} > 0
\end{equation}
\end{lem} 
\begin{pf}

This follows from the relation \ref{change} above along with the Ecker-Huisken maximum principle as utilized by Bernstein and L. Wang; see lemma 3.2 of \cite{BW}.
\end{pf}

\subsection{The renormalized mean curvature flow with (localized) surgery}
$\medskip$

Our goal is to construct an inward level set flow $L_t$ corresponding to $M \in \Sigma$ by the RMCF. To do that we will start by constructing a mean curvature flow with surgery out of $M$. In this section for the sake of brevity we assume familiarity with Haslhofer and Kleiner's mean curvture flow with surgery \cite{HK} in Euclidean space as well as its adaption to the general ambient setting by Haslhofer and Ketover \cite{HKet}, but below we will give summaries when possible to orient the reader. 
$\medskip$

Giving a very brief account of the surgery flow, recall that in the mean curvature flow with surgery one finds for a mean convex surface $M$ (in higher dimensions, 2-convex) curvature scales $H_{th} < H_{neck} < H_{trig}$ so that when $H = H_{trig}$ at some point $p$ and time $t$, the flow is stoped and suitable points where $H \sim H_{neck}$ are found to do surgery where ``necks,'' high curvature regions where the surface will be approximately cylindrical, are cut and caps are glued in. The high curvature regions are then topologically identified and discarded and the low curvature regions will have mean curvature bounded on the order of $H_{th}$. The flow is then restarted and the process repeated. 
$\medskip$

There are a couple different approaches on the construction of the mean curvature flow with surgery; aside from Haslhofer and Kleiner's approach see the work \cite{HS} of Huisken and Sinestrari for the original paper on MCF with surgery for $n \geq 3$ and the paper of Brendle and Huisken \cite{BH} for its extension to $n =2$. Here we will follow Haslhofer and Kleiner as all their results are stated locally. There the curvature thresholds are determined by the parameters $\Balpha = (\alpha, \beta, \gamma)$, fixed real numbers. Here $\alpha$ is a noncollapsing constant: we say a surface is $\alpha$ noncollapsed if there are inner and outer osculating balls of radius (at least) $\alpha/H$. Andrews, Sheng and Wang, and White \cite{BA,SW,W} independently showed this is preserved under the MCF (or at least that its implied by their work). $\beta$ is a $2$-convexity assumption which for our case is set to 1 (since we are only involved with surfaces in $\R^3$), and $\gamma$ is an initial bound on mean curvature. 
 $\medskip$
 
 For our purposes, we will replace the role of $H$ with $F = H - \frac{\langle X, \nu \rangle}{2}$ and say surfaces which are noncollapsed with respect to $F$ are $F$ $\alpha$-noncollapsed; recall from above that convexity of $F$ with respect to the renormalized flow is preserved. For motivation we next discuss $F$-noncollapsing under the RMCF in just the compact case:
 \begin{lem}\label{noncollapsing} Suppose $M$ is a compact manifold which is $F$ $\alpha$-noncollapsed and consider $M_t$, the flow of $M$ under the renormalized mean curvature flow. Then there is a continuous function $C(t) > 0$ depending on $\alpha$ only with $C(0) = \alpha$ for which $M_t$ will be $F$ $\alpha$-noncollapsed with with constant $C(t)$. 
 \end{lem} 
 \begin{pf}
 In remark (7) of \cite{BA} Andrews notes that noncollapsing is preserved under the (regular) mean curvature flow for positive functions $f$ satisfying $\frac{df}{dt} = \Delta f + |A|^2f$; see also \cite{ALM} for more general homogeneous flows and \cite{Lin} for general Haslhofer-Kleiner type curvature estimates. Noting that $f = -tH - \frac{\langle X, \nu \rangle}{2}$ is such a function in our setting due to Smoczyk \cite{Smo}, noncollapsing with respect to $f$ is preserved under the MCF on $[-1, T)$. Using that for any interval $[-1, T)$, $T < 0$, that the distortion in the continuous reparameteriztion \ref{param} is bounded, we see that balls will not be mapped to points. In particular within these regions then we can find osculating balls with diameter bounded below in terms of the original ones depending on $t$,  giving us the statement because $f$ is brought to $F$ under the coordinate change up to a positive, continuous factor which is also bounded on $[-1, T)$.  \end{pf}

We will localize the mean curvature flow with surgery much as in the spirit of the author's previous work \cite{Mra}; we first remark that a version of the pseudolocality theorem holds for the RMCF via the reparameteriztion \ref{param}:
\begin{lem}\label{plends} Let $M \in \Sigma$ and consider its RMCF $M_t$. For any $\epsilon, T > 0$ finite there exists $R_1$ such that for any ball $B(p,r) \subset B(0, R_1)^c$, $|A| < \epsilon$ on $M_t \cap (B(p,r) \times [0, T])$. 
\end{lem}

Below we will refer to an application of lemma \ref{plends} by a slight abuse of terminology as pseudolocality. With this in hand we now discuss how to define a mean curvature flow with surgery on elements in $\Sigma$: 
\begin{thm}\label{surgery} For any $M \in \Sigma$, there is a flow with surgery $S_t$ starting from $M$, defined on $[0, \infty)$, which agrees with the renormalized mean curvature flow except for a discrete set of times $t_i$ at which necks are cut and replaced by caps. The surgery parameters depend on time, but on each interval $[k, k+1]$ there is a uniform choice of parameters $F_{th,k}$, $F_{neck,k}$, $F_{trig,k}$ which suffice. Additionally one may take $F_{th,k}$ arbitrarily large if they choose. 
\end{thm}
\begin{pf}

Reiterating the above, for mean convex surfaces in $\R^3$ the curvature scales $H_{th}, H_{neck}, H_{trig}$ depend on an $\alpha$ noncollapsing constant and initial bound on $H$; in the notation for surgery parameters above we replaced $H$ with $F$ indicating our use of the quantity $F = H - \frac{\langle X, \nu \rangle}{2}$. As we discussed above in the compact case $\alpha$ noncollapsing with respect to $F$ is preserved with some possible deterioration in the constant for compact noncollapsed surfaces (at least using our simple reasoning); we face the added difficulty of noncompactness though and, since $F \to 0$ at the ends, there may be no choice of $\alpha$ for which our $M \in \Sigma$ is $\alpha$-noncollapsed as well. 
$\medskip$

We will deal with this issue of noncollapsing by localizing it where it is needed. We will say a surface $M$ is $\alpha$-noncollapsed in a ball $B$ if for any $x, y \in M \cap B$ the $x$ (resp $y$) is not in either the inner or outer osculating ball at $y$ of radius $\alpha/F(y)$ (resp $x$). Let $M \in \Sigma$, fix $T > 0$ and suppose its smooth flow exists on $[0, \epsilon)$. By pseudolocality, one may choose $B(0,R)$ large enough so that any singularity which may occur lays within $B(0,R)$ (iterating the following, up to time $T$). Recalling from lemma \ref{con} above the decay rate of $F$ is bounded below on the ends for finite times so that in a sufficient large annulus $A= B(0,2R) \setminus B(0,R)$ $F > c$ on $[0, \epsilon)$ and hence the surface is $F$ $\alpha$-noncollapsed for some $\alpha_0$ in the annulus $A$. Switching to the corresponding regular MCF and denoting momentarily $\widetilde{F} = -tH - \frac{\langle X, \nu \rangle}{2}$, $\widetilde{\epsilon}$ for $t^{-1}(\epsilon)$, and similarly defining $\widetilde{M}_t$, $\widetilde{A}$ we get that $\widetilde{M}_t$ is $\widetilde{F}$ noncollapsed in $\widetilde{A}$ on $[-1, \widetilde{\epsilon})$ for some $\widetilde{\alpha_0}$. By applying the maximum principle to the function $Z(x,y,t)$ from Andrew's proof \cite{BA} of noncollapsing (see proposition 3.2 in the author's previous article \cite{Mra}), the noncollapsing constant extends into the inner ball bounded by the annulus. Switching back to the RMCF $M_t$ gives that the noncollapsing extends into $B(0,R)$ for some constant $\alpha_1$ arguing as in lemma \ref{noncollapsing}. 
$\medskip$

This allows us to employ the mean curvature flow with surgery within $B(0,R)$ exactly as in section 7 of the paper of Haslhofer and Ketover \cite{HKet}, where the mean curvature flow with surgery is developed for curved ambient spaces. More or less along these same lines, section 4 of the author's previous joint work with S. Wang \cite{MW2} also applies by reversing the parameterization and working in the MCF as in the above paragraph. The point is that one can show curvature estimates analogous to the standard case in $\R^{n+1}$ for $F$-noncollapsed flows and blowup limits will be noncollapsed mean convex MCFs. This leads to global convergence and convexity theorems which lead to a canonical neighborhood theorem. 
$\medskip$

Noting that the surgery can be arranged so that if a surgery is done at a time $T_s < T$ the noncollapsing constant obtained still holds,  in particular from the above we get that on each interval $I_k = [0, k]$ a surgery flow with constants $F_{th,k} < F_{neck,k} < F_{trig,k}$ as described above may be performed, although as $k \to \infty$ the surgery parameters would be expected to degenerate so a uniform choice is not guaranteed to hold. To deal with this small issue, note that the same argument gives on $[k, k+1]$ uniform noncollapsing within the regions high curvature may develop, so we can simply readjust the surgery parameters on each interval $[k, k+1]$ to get a surgery flow with discrete surgery times. We will denote it by $S^k_t$ for the interval $[k, k+1]$, when we are in situations where it is useful to refer to the (changing in time) surgery parameters and $S_t$ when we are not. The final comment in the statement above is due to the fact that actually the ratios of the surgery parameters just have to be bounded below by a certain amount. \end{pf}

\subsection{The level set flow of elements of $\Sigma$.} 
$\medskip$

The point of this section is to prove theorem \ref{converge} below, that the surgery flows above converge to the level set flow of $M$ and it is nonfattening. An immediate corollary of the first part will be items (1) and (2) of theorem \ref{LSF}, while that the flow is nonfattening will be used in the sequel. Recalling the definition of (set-theoretic) weak and level set flows by Ilmanen \cite{I1} for the regular mean curvature flow, a weak set flow is a family which satisfies the avoidance principle:

\begin{defn}[Weak Set Flow]\label{original_defn} Let $W$ be an open subset of a Riemannian manifold and consider $K \subset W$. Let $\{\ell_t\}_{t \geq 0}$ be a one -parameter family of closed sets with initial condition $\ell_0 = K$ such that the space-time track $\cup \ell_t \subset W$ is relatively closed in $W$. Then $\{\ell_t\}_{t \geq 0}$ is a weak set flow for $K$ if for every smooth closed surface $\Sigma \subset W$ disjoint from $K$ with smooth MCF defined on $[a,b]$ we have 
\begin{equation}
\ell_a \cap \Sigma_a = \emptyset \implies \ell_t \cap \Sigma_t = \emptyset
\end{equation} 
for each $t \in [a,b]$
\end{defn} 
Note this defintion can be applied to very general intial data, including domains as well as closed smooth hypersurfaces. The set theoretic level set flow is the largest weak level set flow: 

\begin{defn}[Level set flow] The level set flow of a set $K \subset W$, which we denote $K_t$, is the maximal weak set flow. That is, a one-parameter family of closed sets $K_t$ with $K_0 = K$ such that if a weak set flow $\ell_t$ satisfies $\ell_0 = K$ then $\ell_t \subset K_t$ for each $t \geq 0$. 
The existence of a maximal weak set flow is verified by taking the closure of the union of all weak set flows with a given initial data. 
\end{defn}

We say the level set flow of a hypersurface fattens when it develops an interior -- this is closely related to uniqueness questions for flows emanating from a given intial datum. We warn the reader that the level set flow of noncompact sets can be more complicated in comparison to the compact case, see section 7 of \cite{I3} for some pathological examples. Thus one must proceed with caution, even in the mean convex case where it is known to be nonfattening when the initial data is compact. 
$\medskip$

 Since the RMCF is a reparameterization of the MCF the avoidance principle still holds; hence one can use the same definitions with respect to the RMCF (actually a so-called super avoidance principle holds, as discussed in \cite{OW}); in fact the level set flow with respect to RMCF can be gotten from the one for the MCF via the reparameterization. We remark there is also the related notion of level set flow due to Evans-Spruck and Chen-Giga-Goto \cite{ES}, \cite{CGG} where they define it via the level sets of viscosity solutions to 
\begin{equation}\label{visc}
w_t = |\nabla w| Div\left(\frac{\nabla w}{|\nabla w|} \right)
\end{equation}

In \cite{I1}, section 10, Ilmanen shows in the compact case these notions are equivalent and this should follow in our particular noncomapct case by the work below but its unnecessary to do so for our goals. Denote by $S^{k,i}_t$ a sequence of surgery flows with curvature thresholds $\{(F_{th,k})_i\} \to \infty$ as $i \to \infty$ for each fixed $k$ with initial data $M \in \Sigma$. When the index $k$ isn't important (and it often isn't for our needs) we will just index above by $i$ and write $S^i_t$; in case a given time $t$ is a surgery time, we mean the surgery flow before surgery is done. From the construction above we see that on any finite interval $[0, T]$ we may suppose these flows are $F$ $\alpha$-noncollapsed for some uniform $\alpha$ in a uniform bounded set outside of which they have uniformly bounded curvature. The work of Lauer \cite{Lau} and Head \cite{Head} in the compact case suggest then the following:

\begin{thm}\label{converge} The Hausdorff limit of $S^{i}_t$ as $i \to \infty$, which we denote to be $M_t$, agrees with the level set flow $L_t$ of $M$ and the level set flow of $L_t$ is nonfattening. 
\end{thm}

Note that the limit of spacetime tracks above exists at least after taking a subsequence by the general fact that the space of relatively compact sets of a complete metric space in a bounded domain is compact under Hausdorff convergence and that the shrinker mean curvature on the end tends to zero uniformly; because the convergence is monotone in terms of bounded sets taking a subsequence can then be seen to be unnnecessary. To show this statement we proceed in the same spirit as in section 4 of the author's previous paper \cite{Mra} which in turn was inspired by the arguments in \cite{HW} and \cite{Lau}, with some simplifications here tailored to the specific setting. Roughly speaking, the idea is to use the approximating surgery flows slightly offset in time as barriers to the level set flow, and use that $M_t$ is defined as a limit of these to argue. 
$\medskip$

Of course, an issue with using the approximating surgery flows as barriers though is that they are noncompact. The following lemma will be used below to show that one may approximate the surgery flows with compact ones sufficiently well so that they may be used however. 
\begin{lem}\label{plite} Suppose $M_1$, $M_2$ are two submanifolds of $\R^N$ such that for each compact domain $K$ there exists $C(K)$ so that Area$(M_i \cap K) < C(K)$, and whose mean curvature flows exist on the finite interval $[0,T]$ and $|A|, |\nabla A|$ are uniformly bounded along them by say $C$. Then given $\epsilon, R > 0$, there exists $R'> R$ so that if $M_1 \cap B(0,R') = M_2 \cap B(0, R')$ then $(M_1)_t \cap B(0,R)$ is $\epsilon$--close in $C^2$ local graphical norm to $(M_2)_t \cap B(0,R)$ for all $t \in [0,T]$. \end{lem} 
\begin{proof}
Without loss of generality $R= 1$. Suppose the statement isn't true; then there is a sequence of hypersurfaces $\{M_{1i}, M_{2i} \}$,  $R_i \to \infty$ and times $T_i \in [0,T]$ so that $M_{1i} = M_{2i}$ on $B(0,R_i)$ but $d((M_{1i})_{T_i}, (M_{2i})_{T_i})_{C^2} > \epsilon$ in $B(0,1)$, satisfying the curvature and area bounds. By passing to subsequences of the flows via Arzela-Ascoli using the curvature and area bounds we get twice differentiable flows $(M_{1\infty})_t$, $(M_{2\infty})_t$ in the limit defined on $[0,T]$ and with $|A|$ uniformly bounded by $C$. Using the finiteness of $T$ we may also pass to a further subsequence to arrange that the times $T_i$ converge to, say, $T_\infty$. We see then that the two flows initially agree but at $T_\infty \in [0,T]$ do not. This is a contradiction to theorem 1.1 in \cite{CY}, where they extend classical uniqueness theorems for the flow to the noncompact setting. 
\end{proof}

\begin{proof}[Proof of theorem \ref{converge}:] By the short time existence of the flow, the classical flow of $M$ is smooth for a short time, say up to time $t = 1$ for concreteness. We first consider the level set flow starting from $M_{1/2}$ then so that one may speak of approximating surgery flows to $M_t$ for short times both after and before a given time $t > 1/2$. For any $0< \delta < 1/2$ we have that the level set flow $\overline{L}_{t}$ of $M_{1/2}$ initially satisfies $S^i_{1/2 + \delta} \subset \overline{L}_0 \subset S^i_{1/2 - \delta}$ where the set inclusion is in terms of associated bounded shrinking/decreasing domains; of course for all $i$ $S^i_{1/2 \pm \delta}$ is simply $M_{1/2 \pm \delta}$.
$\medskip$

Now on a given bounded time interval $[0,T]$ the surgeries of the $S^i_t$ as discussed may only appear in a bounded ball $B \subset \R^3$, and outside $B$ the flows will be smooth with uniformly bounded curvature which we may arrange to be as small as we wish by expanding $B$. Furthermore we note that in $M \cap B^c$ every point has a neighborhood of diameter bounded uniformly below where $M$ can be written as a single graph. Lastly, we note even in $B$ for a fixed $i$ the curvature $A$ and its derivatives will be unifomly bounded in terms of ${F_{neck}}_i$ (again suppresing the index $k$). 
$\medskip$
 
Now consider $\epsilon > 0$ given by the minimum of the first surgery times of $S^i_{1/2 + \delta}$ and $S^i_{1/2 - \delta}$ for some choice of $i$ to be determined. By capping the surfaces appropriately we may approximate $S^i_{1/2 + \delta}$ and $S^i_{1/2 - \delta}$ in a ball of radius $j$ centered at any given point by smooth closed compact surfaces disjoint from $\overline{L}_0$ so that they agree with the original surface in a ball of radius $j/2$. By the global curvature bounds on $S^i_t$ for $t \in [0,T]$ and the second point above (graphcial in uniformly large balls) these can be arranged to have uniformly bounded $|A|, |\nabla A|$ so that their flows will exist smoothly for some uniform forward time $\epsilon_1$ which, at least just using crude estimates with the evolution equation for curvature, is potentially less than $\epsilon$. Similarly one may estimate their entropy to see they will also have uniformly bounded area in any given compact set on $[0,T]$ by monotonicity, so by lemma \ref{plite} then on the time interval $[0, \epsilon_1]$ $\overline{L}_t$ will stay disjoint from $S^i_{1/2 + \delta + t}$ and $S^i_{1/2 - \delta + t}$. If $\epsilon_1 < \epsilon$, we may then iterate the argument so without loss of generality $\epsilon < \epsilon_1$.  We next consider what happens when the surgery is performed, arguing much as in Lauer \cite{Lau}. 
$\medskip$

Indeed, considering points arbitrarily far from the origin the distance between $L_t$ and the surgery flows must tend to zero but one can get more refined information using Illmanen's localized avoidance principle with the (explicitly known) initial data in a suitably large ball containing $B$. By it and lemma \ref{con} (see Theorem C.3 in \cite{CCMS}; surgery flows are weak set flows between surgery times) the distance between $\overline{L}_t$ and the surgery flows is bounded below within $B$ on $[0, \epsilon]$ by, say, $\sigma > 0$ independent of surgery parameters. With this in mind for $t = \epsilon$ one can see as in proposition 2.2 of \cite{Lau} that if $i$ is large enough $\overline{L}_t$ will remain disjoint from the surgery flows postsurgery and remain distance $\sigma$ away from them in $B$: the point is the surgery caps are placed on high curvature necks, which $\overline{L}_t$ must not lay within by the distance lower bound. The argument can then be iterated to see that for $i$ large enough depending on $\delta$ and $T$, $\overline{L}_t$, $S^i_{1/2 + \delta + t}$, and $S^i_{1/2 - \delta + t}$ will remain disjoint on $[0,T]$, where the dependence on $T$ is because $\sigma$ would need to be adjusted when applying localized avoidance. 
$\medskip$

Denote by $K_t$ the associated motion of the domain bounded by $M_t$ that the flow is into, and similarly $K^i_t$ for $S^i_t$. From the above we have $\overline{L}_t \subset \lim\limits_{\delta \to 0^+} K_{1/2 - \delta + t} \setminus K_{1/2 + \delta + t}$. Now for each $\epsilon > 0$ there exists $i_0 >> 1$ so that for a fixed $i > i_0$, one can take $\delta$ small enough so that  $K^i_{1/2 - \delta + t} \setminus K^i_{1/2 + \delta + t}$ doesn't contain a ball of radius greater than $\epsilon/3$ from the reasoning above using \cite{Lau} and curvature bounds along a fixed surgery flow. By the Hausdorff convergence, there exists $i_1 >> 1$ so that for $i > i_1$ the distance between the spacetime track of this set and that of $K^i_{1/2 - \delta + t} \setminus K^i_{1/2 + \delta + t}$ is less than $\epsilon/3$, irrespective of choice of $\delta$. So, from considering a fixed $i$ greater than $i_0$ and $i_1$ we see that there must be $\delta > 0$ so that the spacetime track  $K_{1/2 - \delta + t} \setminus K_{1/2 + \delta + t}$ doesn't contain a ball of radius greater than $\epsilon$. Since when varying $\delta$ these sets are nested we see that $\lim\limits_{\delta \to 0^+} K_{1/2 - \delta + t} \setminus K_{1/2 + \delta + t}$ has empty interior as a subset of $\R^{n+2}$, implying that $\overline{L}_t$ is nonfattening arguing by contradiction and the comparison principle. Its easy to see that $M_{1/2 + t}$ is a weak set flow of $M_{1/2}$ itself so must be contained in $\overline{L}_t$. Using that $M_{1/2 + t}$ is a weak flow similarly by a barrier argument one can see for any fixed $t$ $\lim\limits_{\delta \to 0^+} K_{1/2 - \delta + t} \setminus K_{1/2 + \delta + t}$ has no interior, so from the \textit{lower} bound on $H - \frac{\langle X,\nu \rangle}{2}$ along all the sugery flows this difference is exactly $M_{t + 1/2}$ giving that $\overline{L}_t$ is contained in $M_{t+1/2}$ so that they are equal. 
$\medskip$

By using compact approximating barriers and lemma \ref{plite} again, this time to $M$ itself, one can see the level set flow, $L_t$, of $M$ agrees with the classical MCF $M_t$ of $M$ for a short time $s > 0$ which can be iterated to $t = 1/2$ using that $M_t$ is smooth with bounded curvature for $t < 1/2$, so that from above $M_t$ defined in terms of limits of surgery flows agrees with the level set flow $L_t$ of $M$ for all $t > 0$. By back parameterizing, the same is true for the ``MCF level set flow'' of $M$ as well. 
\end{proof}

\begin{rem}
Note that the surgery flows are locally graphical in uniformly large extrinsic balls in $B^c$ play an important role with example 7.3 in \cite{I3} in mind; this was used when we considered compact approximators to the surgery flows which had uniformly bounded curvature and essentially they serve as barriers to keep mass from ``rushing in'' from spatial infinity.
\end{rem}

\newpage

\subsection{Checking $M_t$ is a forced/renormalized Brakke flow}
$\medskip$

In this section we check item (3) of theorem \ref{LSF}, completing its proof. In this section and and the next we assume some familiarity with the Brakke flow or at least White's version of the Brakke regularity theorem \cite{W1}; we first recall the definition of Brakke flow:

$\medskip$

\begin{defn} A (n-dimensional integral) Brakke flow is a family of Radon measures $\mu_t$ such that:
\begin{enumerate}
\item For almost every $t \in I$ there exists an integral $n$-dimensional varifold $V(t)$ so that $V(t)$ has locally bounded first variation and has mean curvature vector $\vec{H}$ orthogonal to Tan$(V(t), \cdot)$ a.e. 
\item For a bounded interval $[t_1, t_2] \subset I$ and any compact set $K$,
\begin{equation}
\int_{t_1}^{t_2} \int_K (1 + H^2) d\mu_t dt < \infty
\end{equation}
\item (Brakke inequality) For all compactly supported nonnegative test functions $\phi$, 
\begin{equation}
\int_{V(t_0)} \phi d\mu_{t_0}  \leq \int_{V(0)} \phi d\mu_0 + \int_0^{t_0} \int_{V(t)} \phi H^2 - H \langle \nabla \phi, \nu \rangle - \frac{d\phi}{dt} d\mu_t dt
\end{equation} 
\end{enumerate}
We will say a Brakke flow is unit regular if it is true for the varifolds $V(t)$ defined above. 
\end{defn}

There are also related notions for Brakke flows with forcing terms, as discussed in \cite{W3}, where when considering flows $\vec{H} + P$ (3.10) there would be an additional term involving $ \nabla \phi \cdot P^\perp \nu - \phi \vec{H} \cdot P$. Here of course we will use $P = X$ and will refer to such flows as renormalized Brakke flows; the coordinate transformation of a smooth Brakke flow gives a smooth renormalized Brakke flow so implies the same is true for the Brakke flow constructed below by Brakke compactness for renormalized flows. The following, which is well known in the compact case, can be checked a variety of ways. For instance perhaps following \cite{Head}, using that $M_t$ is a limit of surgery flows which in a sense are ``nearly'' Brakke flows. Below, we use an argument borrowed from section 7 of Choi, Chodosh, Mantoulidis, and Schulze \cite{CCMS}.
$\medskip$

\begin{lem}\label{Brak}  $M_t$ is a unit regular integral renormalized Brakke flow on $[0, \infty)$. 
\end{lem}
\begin{pf}
 Fixing a time interval $[- 1, \tau]$, where $\tau < 0$, and denoting by $\widetilde{M}_t$ the back-parameterized flow (from the RMCF to the the MCF) of $M_t$, we first note arguing as in proposition 7.3, lemma 7.4 and the following comments of \cite{CCMS} there is a (unit regular integral) Brakke flow whose support $\widetilde{B}_t = \widetilde{M}_t$ on $[-1, \tau]$. This is constructed by approximating the initial data $\widetilde{M}$ by compact hypersurfaces whose level set flow does not fatten, which is possible since nonfattening of the level set flow for compact hypersurfaces is a generic property, so that their level set flows correspond to/are the supports of Brakke flows by elliptic regularization. Taking a limit by Brakke compactness, we get a Brakke flow with initial data $\widetilde{M}$. 
 $\medskip$

In particular, and this will be used in the sequel, the elliptic regularization proceedure applied to the compact approximating hypersurfaces shows that there is a Brakke flow $\widetilde{N}_t$ out of $\widetilde{M} \times \R_{x > -1} \subset \R^{4}$ which splits off an open line segment $(-1, \infty)$ in the $x_{4}$ direction for all times (note we could pick any lower bound on this line segment by translation) and is a limit in any cylinder over a compact set $\widetilde{K} \in \R^{3} \times \{0\}$ of smoothly translating flows which along the limit are increasingly close to being cylinders in the $x_{4}$ direction over $N_t \cap \{x_{4} = 0\}$ in a uniform neighborhood of the $x_{4} = 0$ plane; they translate upwards in the $e_{4}$ direction and the tip of these translators lay well below the plane sufficiently far along the sequence for a fixed finite time interval. Note in the following that the rescaling factor in the renormalization change of coordinates is greater than 1, so the open cylinders above and domains won't be crushed down by the coordinate change. Fixing a compact set $\widetilde{K}$ and denoting the corresponding approximating sequence of smooth flows by $\widetilde{N}^i_t$, we have by applying the renormalization change of coordinates to them and taking a limit using Brakke compactness for renormalized flows that there is a renormalized Brakke flow $N_t$ of $M \times \R_{x > -1}$ on $[0, T]$ where $T$ is $\tau$ under change of coordinates. It splits off an open line segment so that restricted to $K$, a compact set which is contained in the image of $\widetilde{K}$ under the change of coordinates on the time interval $[-1, \tau]$, it is a limit of smooth RMCFs $N^i_t$ which are nearly vertical in the $x_{4}$ direction near the $x_{4} = 0$ plane; we emphasize this last fact because it will be used in the next subsection when discussing Brakke regularity. Taking the $x_{4} = 0$ slice of the $N_t$ gives a renormalized Brakke flow $B_t$ with initial data $M$. By the nonfattening of the level set flow of $M$ as shown in the last section, the support of this flow is contained in $M_t$. From a barrier argument similar to the one in the section above applied to the approximating flows, along with that they are cannonical boundary motions (so their supports are all boundaries of evolving sets), they are in fact equal. Taking $\tau \to 0$ and hence $T \to \infty$ gives the claim. \end{pf}

\subsection{Long term fate of the flow}\label{ltf}
$\medskip$

In this section we wish to show we may apply some deep results of White to the flow, mainly from \cite{W}. As above, we denote by $K_t$ the associated motion of the domain $M$ flows into, so that $M_t = \partial K_t$. Recall now the Brakke regularity theorem, which roughly says that if in a spacetime neighborhood of a point the Gaussian densities of a Brakke flow are close to one the flow is smooth with bounded curvature. White in \cite{W1} showed this was also true in many circumstances for Brakke flows with a forcing term. In particular, the construction above, since the flow is a limit of ``nearly vertical'' smooth flows in a neighborhood of $x_{n+2} = x_4 = 0$ over any compact domain, gives that White's Brakke regularity theorem applies to $M_t$ where the parameters at a point in the statement of Brakke regularity theorem depend on distance to origin -- in particular, see White \cite{W1}, sections 4 and 7. 
$\medskip$

 A shrinker mean convex/inward RMCF level set flow gives a mean convex foliation in the Gaussian metric by \ref{Hrelation} so, amongst other results of White's work holding for our flow, therefore satisfies the one sided minimization property of White with respect to the Gaussian area. This gives the results of section 3 of \cite{W} hold -- of course, with minor modifications of the statements to reflect we are working in the RMCF.  Also since the flow is shrinker mean convex, blowups at any point will be mean convex in the traditional sense. 
 $\medskip$
 
Since $M_t$ is the RMCF level set flow as discussed in the previous section, $K_t$ is a nested weak set flow so the expanding hole theorem and the rest of section 4 of \cite{W} holds. Sections 5,6,7 of \cite{W} hold because the flow is shrinker mean convex, corresponds to a Brakke flow, and satisfies one sided minimization as discussed above. 
$\medskip$

As a consequence of the above sections White's sheeting theorem in section 8 of \cite{W} holds to see that near any multiplicity 2 tangent plane, the flow can be written as the union of two graphs. In particular, by dimension reduction static and quasistatic limit flows are smooth (since we are considering flows in dimension less than 7). Compare also with proposition 8.6 to lemma 8.8 of \cite{CCMS}. With all this in mind it follows one may apply section 11 of \cite{W} to see that:
\begin{thm}\label{minlimit} If $\lim\limits_{t \to \infty} M_t$ is nonempty, the limit is a smooth stable self shrinker. 
\end{thm} 
In more detail concerning the claim of stability let's denote the limit shrinker by $N$; note it will be of bounded entropy and hence polynomial volume growth. In any bounded intrinsic ball $B$ it will be stable with respect to variations supported in that ball by White's theory and in particular the lowest Dirichlet eigenvalue of the Jacobi operator on $N \cap B$ is bounded below by 0. As discussed in the proof of lemma 9.25 of \cite{CM} the lowest eigenvalue $\lambda_1$ of the Jacobi operator on $N$ is the limit of them on an exhaustion of $N$ by balls implying $\lambda_1 \geq 0$, giving a contradiction considering its no more than $-1/2$ as also shown in that paper, specifically theorem 9.2. We note here that in the eventual application of this fact below one could instead use the Frankel theorem for shrinkers to finish, so stability and even smoothness of the limit actually doesn't seem to be crucially important.  
$\medskip$

\begin{rem} We point out one should indeed expect singularities should occur. For example an outward perturbation of the Angenent torus will develop under RMCF a neckpinch about its axis of rotation -- in this case, the RMCF will then flow outwards to spatial infinity, a consequence of the Angenent torus being unknotted. Hence it seems necessary to consider a weak flow as above -- see also the recent work of of Lin and Sun \cite{LS} which says this is the case for all compact nongeneric shrinkers. 

\end{rem}

\section{Proof of Theorem \ref{thm:unknotted}}

Our goal as in \cite{MW1} and \cite{L} is ultimately to appeal to a Waldhausen type theorem \cite{W}, perhaps the best known formulation being the following:  
\begin{thm}\label{W} Suppose $M$ is a Heegaard splitting of $S^3$ of genus $g$. Then it is isotopic to the standard genus $g$ surface of $S^3$.
\end{thm} 

A Heegaard splitting is a surface in a 3 dimensional (for now, take it to be comapct) manifold $N$ which splits $N$ into two handlebodies: regions homeomorphic to topologically closed regular neighborhoods of properly embedded, one-dimensional CW-complexes in $N$. We define standard embeddedness for compact closed surfaces as surfaces isotopic to any of the following. The standardly embedded torus we take to be the embedding $T^2\hookrightarrow \R^3 \sim S^3 \setminus \{\infty\} \hookrightarrow S^3$ given by rotating the unit circle $S(2,1)$ in the $xy$ plane about the $z$ axis. The standardly embedded genus $g$ surface can be constructed by taking $g$ standardly embedded tori, arranging so that their centers fall along a line and so that their convex hulls are pairwise disjoint, and taking a connect sum of adjaicent tori using straight cylinder segments at two closest points. 
$\medskip$

 In our noncompact case, we will say that a one ended surface is standardly embedded if there is an isotopy which takes it to the connect sum with a standardly embedded genus $g$ surface attached to a plane, in agreement with Frohman and Meeks \cite{fm}; in particular see figure 1 and the surrounding discussion in \cite{fm}. Indeed such an isotopy would have to have ``infinite speed;'' for a simple example why this is necessary consider defining an isotopy from a conical to a planar end. However, below it will be clear no isotopies which unknot by sending problematic regions of $M$ to spatial infinity will be used. 
$\medskip$

Focusing momentarily on just the closed case, to use Waldhausen's theorem one needs conditions for when a surface is a Heegaard spliting. Lawson in \cite{L} gives the following criteria, where (2) is particularly useful for verfication using ideas from geometric analysis. 
\begin{lem}\label{equiv} Let $M$ be a closed hypersurface in $S^3$ and denote by $R_{in}$ and $R_{out}$ the inner and outer regions bounded by $M$. Then $M$ is a Heegaard splitting exactly when either (and hence both) of the two statements in the following is true:
\begin{enumerate}
\item  The inclusion maps $\iota: M\to R_{out}$,  $\iota: M\to R_{in}$ both induce surjections of fundamental groups $\iota_* : \pi_1(M) \to \pi_1(R_{out}), \pi_1(R_{in})$. 
\item $\widehat{R_{out}}$ and $\widehat{R_{in}}$, where $\widehat{R}$ denotes the universal cover, have path connected boundary. 
\end{enumerate}
\end{lem}

The labeling of one region as the outer one and the other as inner of course is arbitrary. Throughout we will refer to $R_{out, in}$, et cetera when we want to discuss the pairs $R_{out}, R_{in}$ simultaneously in the fashion that the argument would apply either using $R_{out}$ or $R_{in}$, which is often (but not always) the case. We will also often refer to the first criterion as ``$\pi_1$ surjectivity'' with respect to a given domain. 
$\medskip$

In \cite{MW1} we compactify $\R^3$ to consider their self shrinker as a hypersurface in $S^3$ so as to apply \ref{W} but in the present case there are ends which makes the state of affairs for Waldhausen type theorems more subtle. To see this note that cutting off a noncompact surface, at least one with suitably well controlled ends, by a large ball the question is closely connected to topological uniqueness problems for Heegaard splitings with boundary of balls defined appropriately. Perhaps suprisingly, there are examples of knotted minimal surfaces with boundary constructed by P. Hall in \cite{Hall} which give in turn topologically nonstandard Heegaard splitings of the three ball. The good news is that these have multiple boundary components; if there is just \textit{one} boundary component then a Waldhausen theorem along with the $\pi_1$ surjectivity criterion holds, see section 2 of \cite{Meeks} as well as \cite{fm}. 
$\medskip$

With this in mind note that if in a ball $B$ $M \cap B$ is a Heegaard spliting and that $M\cap B^c \sim \R^2 \setminus D(0,1)$ its easy to see $M$ is standardly embedded by isotoping $M \cap B^c$ to be planar after applying Waldhausen's theorem within the ball and the isotopy extension theorem along the ball's boundary. Its also easy to see when $B$ is large enough that $M\cap B^c \sim \R^2 \setminus D(0,1)$ the inclusion map of $M \cap B$ into either of its bounded components will be surjective on fundamental groups if it is so for $M$, so in summary: 

\begin{lem}\label{importantrem} Lawson's criteria given in lemma \ref{equiv} above can be used to show an asymptotically conical surface with one end is topologically standard. 
\end{lem} 

From here on out denote by $M$ a self shrinker which either has a single, asymptotically conical, end or is compact. First we consider the asymptotically conical case and afterwards we discuss the case it is compact.
$\medskip$

\subsection{\textit{M} has an asymptotically conical end}

Suppose that at least one of $\widehat{R_{out,in}}$ has disconnected boundary (which one in particular is unimportant), so that it has at least two path components $A$ and $B$.
$\medskip$

By lemmas 4.1 and 4.2 of Bernstein and L. Wang \cite{BW} one may find, switching choice of normal depending on which domain $R_{out, in}$ is in question, a shrinker mean convex perturbation of $M$ to $M^{\epsilon} \in \Sigma$ the set defined in section 3. In short it is a perturbation which, on each end, asymptotes to the original asymptotic cone using the first eigenfunction of the Jacobi operator. This perturbation will lay entirely within the domain we chose and will be entropy decreasing. 
$\medskip$

Theorem \ref{LSF} then gives us a renormalized mean convex/inward level set flow $M_t$ with intial data $M^\epsilon$ which exists for all time (by a small abuse of notation), and if it is nonempty the limit will be a stable self shrinker. To tell if it is nonempty we will consider the flow in the universal cover which we now discuss. To begin we see one may lift a submanifold/varifold $N$ of $R_{out, in}$ to the lift $\widehat{N}$ of it in $\widehat{R_{out,in}}$ by, for every open set $U$ in an open covering of $\widehat{R_{out,in}}$ on which the projection map $\pi$ is locally a diffeomorphism, defining $U \cap \widehat{N} = U \cap \pi^{-1}(\pi(U) \cap N)$. This is well defined and by similarly lifting the outward normal gives a way to lift perturbations and flows -- if the lift of $N$ has multiple components $C_i$ one can also wish to define the flow out of a single component $C$ by lifting the flow of $N$ and throwing out the components which did not originally emanate from $C$.  For surgery flows choosing components to follow is well defined because if two originally disjoint boundary components became conjoined at a later time it must be due to a surgery, but this can't be the case since a surgery and hence its lift for fine enough parameters can only disconnect. Hence it is well defined for the level set flow as well. Also, we see by our proceedure that if $N$ is embedded and boundaryless so is $\widehat{N}$. 
$\medskip$

With that in mind to show the flow of  $M^\epsilon$ will be nonempty we consider the lifts of the perturbation $M^\epsilon$ to the universal cover of $\widehat{R_{out, in}}$ to get a graphical perturbation $A^\epsilon$ of $A$ for $\epsilon$ small -- this is well defined in that there is indeed a connected component of the lift of $M^\epsilon$ for each connected component of $\widehat{M}$, graphical over their component, as a consequence of the fact that contractible open sets are evenly covered and that $M$ has a uniform tubular neighborhood (measured in Euclidean distance) since its asymptotically conical. Similarly, $A^\epsilon$ is embedded for $\epsilon$ small. Furthermore, we may consider the lifted approximating surgery flows $\widehat{S_t}$ and $\widehat{M_t}$ of $M_t$ in $\widehat{R_{out, in}}$ which flow ``out'' of $A$ -- we need not consider the lift of the flow to the other boundary components as discussed above. We now discuss some properties of these lifted flows:
$\medskip$

\begin{lem}\label{props} Any lifted approximating flow $\widehat{S_t}$ with fine enough surgery parameters ($F_{th,k}$ large enough (for each $k$)) and hence $\widehat{M_t}$ satisfies the following properties:
\begin{enumerate}
\item $\widehat{S_t}$ will never collide with a boundary component of $\widehat{R_{out,in}}$
\item $\widehat{S_t}$ and hence $S_t$ is nonempty for all $t \in [0, \infty)$ 
\item Suppose $\widehat{R_{out,in}}$ has at least two boundary components $A$, $B$ and that $\widehat{M}_t$ flows out of $A$. Then for any curve $\gamma$ betewen $A$ and $B$ which has nonvanishing mod 2 intersection number with $A^\epsilon$, $\widehat{S_t} \cap \gamma \neq \emptyset$ for all $t \in [0, \infty)$ so that $S_t$ and hence $M_t$ will have a nonempty limit as $t \to \infty$. 
\end{enumerate}
\end{lem}

\begin{pf}
We focus our discussion on a fixed surgery flow $S_t$ which implies the same facts for $M_t$ by theorem \ref{LSF}. Item (1) follows from the surgery flow $S_t$ being monotonically inward into its respective domain $R_{out, in}$ and the definition of $\widehat{S_t}$ in terms of lifts of $S_t$. Item (2) is clearly a consequence of (3) but we highlight it because of its importance. Considering item (3) then, in the following we consider throughout generic times when the intersection with $\gamma$ is transverse, or alternately slightly deform $\gamma$ as long as one always stays in a fixed neighborhood of the original curve.  We first note that when $\widehat{S_t}$ flows by the (lifted) smooth renormalized mean curvature flow that the mod 2 intersection number is preserved following the same proof that it is preserved under isotopy for two compact closed surfaces. To see this, noting that how the lifted flow was defined $\widehat{S_t}$ will always be embedded because $S_t$ is, this is because $\gamma$ is compact (intuitively, so that intersection points are not ``lost'' to spatial infinity), by (1) that $\widehat{S_t}$ is isolated from the endpoints of $\gamma$, and that $\widehat{S_t}$ is boundaryless. These facts force the spacetime track of the intersection points to be spatially bounded intervals or closed loops so that the mod 2 intersection number is preserved. It then remains to consider how the intersection number may change during surgery times.

\begin{rem} Alternately it seems such situations could possibly be avoided altogether by perturbing the curve $\gamma$ slightly because the singular set of $M_t$ is negligible but we deal with this possibility in a more direct fashion below. 
\end{rem}

To proceed, we must first describe in more detail what could unfold during a surgery time $t^*$ for the unlifted flow $S_t$. Modulo an extra case related to case (iii) this discussion applies equally to $S_t$ and $\widehat{S_t}$ as discussed shortly after:
$\medskip$

 If $S_{t^*}$ has high curvature everywhere, it is either i) convex, ii) close (in appropriate norm, see after remark 1.18 \cite{HK0}) to a tubular neighborhood of some open curve with convex caps, or iii) close to a tubular neighborhood of a closed curve. In these cases the surface is either a sphere or a torus. If there are low curvature regions on a connected component of $S_{t^*}$, then there are couple cases for the high curvature regions it may border. Considering a given high curvature region bordering a low curvature one, there will be a neck, which is a region where at every point after appropriate rescaling the surface is nearly cylindrical, which following along the direction of its axis away from the original low curvature region one will find either a) a convex cap or b) another low curvature component of surface. In the former case, there will be one surgery spot and in the latter there will be two on either side of the neck region and hence four caps will be placed, so that the capped off neck is topologically a sphere. This discussion is encapsulated in figure \ref{fig1}.
$\medskip$

\begin{figure}\label{fig1}
\centering
\includegraphics[scale = .65]{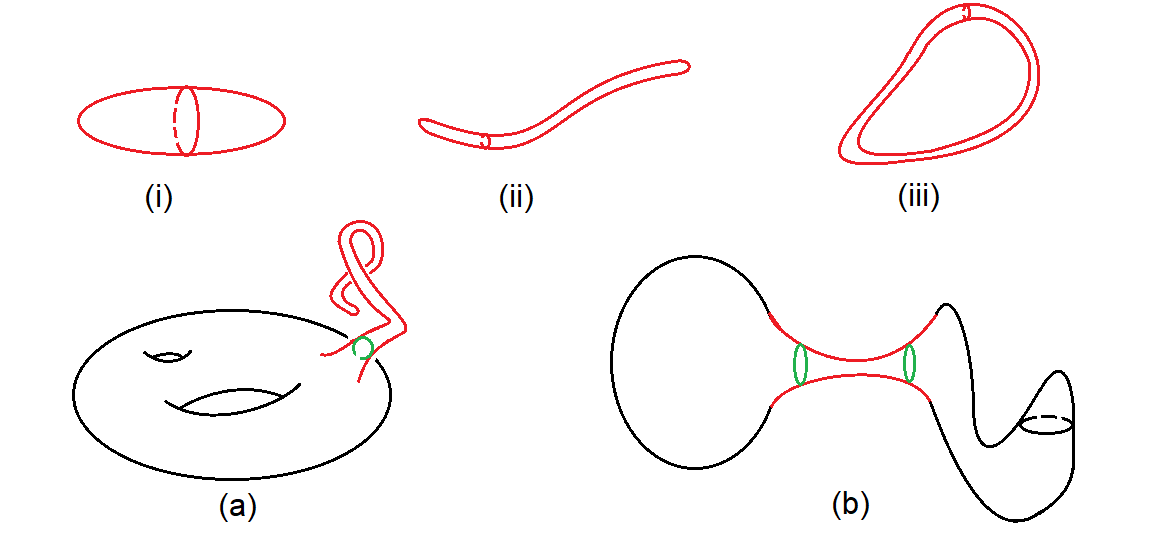}
\caption{A diagram displaying the possibilities one may encounter at a surgery time for a compact flow in $\R^3$. High curvature regions are in red and surgery spots are in green. Note in (b) there are two spots along the neck where surgery will be done giving two pairs of caps.}
\end{figure}

When the surgery parameters are sufficiently fine most of the above discussion applies to the lift of the surgery flows as well, because surgeries which occur at any given finite time take place in compact regions about the origin where as such points within will have a uniform lower bound on the diameter of their evenly covered neighborhood (or, using that generally any contractable open set in the base is evenly covered), so if the surgery parameters are fine enough/$F_{th,k}$ large enough, $k$ depending on the time interval under consideration, all local models -- necks, caps, high curvature convex regions discussed above will be lifted to necks, caps, and convex regions in the universal cover using the local models of each will be completely contained in such a neighborhood.
$\medskip$

Going to the cover there is possibly one extra case though: note that if case (iii) occurs along the surgery flow in the base then the bounded loop may concievably be lifted to an embedded cylinder in the cover, which we refer to below as case (iv). No other new cases can be ``added'' by the covering map since, following the central curves from one side of the neck region (if any) in the base in the other cases and considering then the lift, if a cap or low surgery region is arrived at in the base this must also be the case following the lift of the central curve in the cover. In other words, the lifts of these types of high curvature regions in the base correspond to the same cases passing to the cover although of course in the cover there may be many copies. 
$\medskip$

With this in mind, consider the very first time a surgery is performed. If $\gamma$ does not intersect any regions where surgeries are performed then there is nothing to do. Suppose now that $\gamma$ does intersect some surgery regions and denote by $U$ an open set which contains all (solid) surgery regions and such that $\gamma$ intersects $\widehat{S_t}$ in $U$ only in points affected by surgery; we see it will suffice to show the parity of the intersection number of $\gamma$ with $\widehat{S_t}$ in $U$ will be preserved across surgeries. Note that without loss of generality we may assume $\gamma$ always intersects $\widehat{S_t}$ transversely, even after surgeries and we will do so below. First suppose $\gamma$ intersects no \textit{future surviving caps}: caps placed by a surgery which border a low curvature region and, in particular, are not part of a component which will then be discarded; in other words $\gamma$ only intersects those components which will be discarded after caps are placed. Denoting by $D_i$ the high curvature regions precapping and $\overline{D_i}$ the high curvature regions postcapping if necessary to do so, since no future surviving caps are intersected $\gamma$ intersects $D_i$ in the same number of points as $\overline{D_i}$, each component of which is boundaryless. Since for cases (i)-(iii) the $\overline{D_i}$ are closed surfaces $\gamma$ intersects each of them in an even number of points. In the possible new case encountered in the cover, case (iv), this is also true because $\gamma$ is an embedded compact interval (note if this weren't the case, it could enter the neck and stay in the core, so that its intersection number was 1). On the other hand in this case $\gamma$ will not intersect $\widehat{S_t}$ in $U$ at all post surgery i.e. after also deleting the high curvature regions. Since the number of intersection points went down in $U$ by an even number, we are done in this case. 
$\medskip$

Now suppose $\gamma$ did intersect some future surviving caps. First we consider those it intersects an odd number of times and focus on one of them, which we'll call $C$. In this case, the other side of the future cap is either a high curvature region diffeomorphic to a sphere or another low curvature region, and we will refer to the high curvature regions discarded in these cases as either a ``horn'' or ``neck'' respectively (in pictures, (a) and (b) in the figure above). After perhaps redefining $U$ we can consider a component $V \subset U$ that is a contractible open set containing precisely the horn or neck in accordance with focusing on $C$, which we may arrange so that the boundary of the future surviving cap(s) lay on $\partial V$. Within this set $\gamma$ may have multiple connected components but it suffices to consider them one at a time by the basic fact that a sum of even numbers is even. So, below when we refer to $\gamma$ we mean to consider a (connected) subinterval of it in $V$ intersecting $C$. Finally, by a slight perturbation without loss of generality it intersects the cap in its interior as well.
$\medskip$

In the horn case the intersection number of $\gamma$ with $\widehat{S_t}$ in $V$ is odd. To see this note since $\gamma$'s intersection number with the future cap itself is odd eventually $\gamma$ goes through the core of the neck region of the horn and doesn't return to the region $R$ in $V$ bounded by $\partial V$ and the future cap (parameterizing the curve appropriately). Since the endpoint of $\gamma$ must not lay in $V$ the subinterval of it eventually meets $\partial V$ on the other side of the horn and the assertion follows. Denoting the horn region by $D$ and the post cap placement horn by $\overline{D}$, $\gamma$ will intersect $\overline{D}$ an even number of times since it is closed. But since $C$ is the cap placed opposite the horn and the intersection number of $\gamma$ with it is odd the number of intersection points in $V$ stays odd, preserving mod 2 intersection number in $V$ (and hence $U$). Applying the same argument for every such future surviving cap which came from a neck bordering a horn and using again a sum of even numbers is even covers this case, in that we checked the mod 2 intersection number is left unchanged.  
$\medskip$

 The neck case is the same if $\gamma$ intersects it at only one future surviving cap, but there is also the case $\gamma$ intersects both bordering caps. In this case there will be two relevant pairs of surgery caps placed, a pair associated to $C$ (where the other cap in the pair is part of a neck which is thrown out) and a ``far pair'' on the opposing low curvature region: the cap on the far pair which survives surgery we'll call the \textit{far opposing cap} $C'$. In (b) in the figure above labeling the four caps 1 through 4 from left to right if $C$ were ``1'' then $C'$ would be ``4.''
$\medskip$

In the case $\gamma$ intersects both $C$ and $C'$ there are two cases: it intersects the neck an odd number of times and $C'$ an even number of times, or vice versa: it can't intersect both an odd number of times or else $\gamma$ would  have a triple junction (i.e. a ``Y'') and it can't intersect both an even number of times because it has an odd intersection number with $C$ and it is boundaryless in $V$: in this case an endpoint would lay within the core of the neck but the flow is strictly separated from $A$ or $B$ as used before. In the first case the mod $2$ intersection number of $\gamma$ with $\widehat{S_t}$ in $V$ presurgery is odd and in the second case even. In the first case the intersection number across $C$ postsurgery is odd, across $C'$ even, so the mod 2 intersection number is preserved. In the second case $\gamma$ will intersect both $C$ and $C'$ an odd number of times, so the total intersection number in $V$ postsurgery is even again preserving intersection parity. Again one applies this same argument at every such cap $C$. 
$\medskip$

Now consider those future surviving caps which (a connected component in $U$) of $\gamma$ intersects an even number of times, again focusing on one of them and denoting it again by $C$. In the horn case, $\gamma$ must intersect $D$ an even number of times because it intersects $C$ an even number of times (so every time $\gamma$ goes past $C$ into $D$, it must come back), so as before the mod 2 intersection number is preserved. In the neck case, one can see arguing as above that $\gamma$ much intersect the far opposing cap and the neck both an even or odd number of times since it intersects $C$ an even number of times, giving that the mod 2 intersection number is again preserved under surgery like before. As above then one applies this at every such cap $C$. Repeating the argument for subsequent surgery times gives us that for any such $\gamma$ $\widehat{S}_t$ will always intersect $\gamma$ in an odd number of points.  \end{pf}

Since the domains $\widehat{R_{out, in}}$ are simply connected and $A^\epsilon$ is a graph over $A$, such a $\gamma$ certainly exists as in the statement of item (3). Note that in the lift any curve from $A$ to itself will intersect $A^\epsilon$ an even number of times so (3) is indeed particular to our case, in that we needed two different boundary components in the lift. For a concrete example, the universal cover of $\R^3 \setminus B^3$ is simply connected and, taking the boundary of the ball to be the shrinking sphere, any outward perturbation will flow away to infinity under the RMCF. Denote by $N$ the stable self shrinker we obtain from theorem \ref{minlimit}. Since $M$ was proper, by Ding and Xin (theorem \ref{proper} above) it had polynomial volume growth and hence finite entropy, so that the entropy decreasing perturbation $M^\epsilon$ does as well; more elementarily it has polynomial volume growth because it has a single conical end. By the monotonicity of entropy for nonfattening level set flow this implies $N$ does as well, which in turn gives $N$ has polynomial volume growth. In particular it must not be a stable shrinker, giving a contradiction and showing $\pi_1$ surjectivity holds with respect to both the inner and outer components of $M$. Since $M$ is asymptotically conical we obtain theorem \ref{thm:unknotted} in this case as indicated in lemma \ref{importantrem}. Alternatively one could apply the Frankel theorem to $N$ and $M$ to gain a contradiction, since they must be disjoint by the set monotonicity (item (1) of theorem \ref{LSF}) of the flow.

\subsection{Revisiting the compact case.} 
$\medskip$

To conclude we discuss how to reproduce the unknottedness theorem for compact self shrinkers more in line with the technique above. First note in this case we may appeal straight to Waldhausen's theorem for Heegaard splitings of $S^3$ after one point compactifying $\R^3$ as discussed in \cite{MW1}, note the isotopy can be arranged to avoid $\{\infty\}$ and hence gives rise to a bounded isotopy in $\R^3$. Using the first eigenfunction of the Jacobi operator as above to get a shrinker mean convex perturbation (for the compact case, see lemma 1.2 in \cite{CIMW}) to then flow; the corresponding flow $M_t$ is then constructed exactly as above --surgery in the compact case is easier from a technical viewpoint because the surgery need not be localized. Similarly the convergence to level set flow and that it is a Brakke flow is already known from previous work. We see at no point did we use the noncompactness of $M$ in the proof of lemma \ref{props}, so we get a nonempty albeit possibly noncompact stable shrinker $N$ in the limit from which we derive a contradiction as before. Alternately, using the compactness assumption, one may also derive a contradiction by theorem 7 in \cite{HW} which says the flow of the perturbation must clear out by a distance estimate.
$\medskip$

\section{Concluding remarks}
We begin our discussion at unknottedness theorems for classical minimal surfaces, which will lead naturally into the other topics mentioned in the introduction. Theorem \ref{thm:unknotted} is very much in the spirit of the various works by, for instance, Freedman, Frohman, Meeks, and Yau on classical minimal surfaces in $\R^3$ -- see the papers \cite{Free, fr, fm, fm2, Meeks, MY}; their papers give an essentially complete answer to the type of question under consideration here for classical minimal surfaces, although the arguments in these papers do not seem to obviously carry over to our setting as we explain. The paper most relevent to our present situation is that of Meeks \cite{Meeks}, where he shows the following:
\begin{thm} Suppose $\langle \cdot, \cdot \rangle$ is a complete metric on $\R^3$ with non-positive sectional curvature. Let $M$ be a complete proper embedded minimal surface in $\R^3$ which is diffeomorphic to a compact surface punctured in a finite number of points. Then
\begin{enumerate}
\item If $M$ has one end, then $M$ is standardly embedded in $\R^3$. In particular, two such simply connected examples are isotopic. 
\item If $M$ is diffeomorphic to an annulus, then $M$ is isotopic to the catenoid. 
\end{enumerate} 
\end{thm}

Note that item (2) in the shrinker context is essentially covered by Brendle in Theorem 2 of \cite{Bren}, mentioned already in the introduction.
$\medskip$

Item (1) has a Morse theoretic proof, where the nonpositive sectional curvature enters via Gauss formula to see that the Gaussian curvature of a minimal surface at any point on $M$ must be negative; this implies the height functions involved have no critical points of index 2 which allows Meeks to show minimal surfaces must bound handlebodies in many situations, allowing him to reduce again to a Waldhausen-type theorem in the case of one end as above. 
$\medskip$

Even ignoring the incompleteness of the Gaussian metric, by calculations of Colding and Minicozzi in \cite{CM} the scalar curvature of the Gaussian metric does not have a sign so neither do the sectional curvatures in this metric, as discussed in section 2 although the regions where the scalar curvature is positive and negative are clearly ``simple'' in that the region where it is positive is a ball (of radius $2\sqrt{\frac{n^2 + n}{n-1}}$). However the examples of P. Hall \cite{Hall} seem to rule out decomposing the surface into different ambient curvature regimes and applying different arguments in each because these boundaries may have multiple boundary components. On the other hand, in contrast to the examples of P. Hall, we do note the subsequent paper of Meeks and Yau \cite{MY}, on complete minimal surfaces with finite topology and multiple ends, reduces to the one ended case in a way which sidesteps any possible pathological behavior -- the fact that the minimal surfaces considered are complete is vital. This suggests our result could possibly be extended to the case of shrinkers with multiple ends, or that perhaps a decomposition indicated in the above paragraph was actually workable. However to arrive at this reduction they use that for classical minimal surfaces with finite topology the ends are parallel (aside from their paper see \cite{HM}), which would need to be checked to prove the analogous statement of their theorem for shrinkers, at least if their approach was followed closely.
$\medskip$

When the metric is Euclidean the main result is also a corollary of Theorem 2 in the same paper of Meeks, where it is shown that minimal surfaces of the same genus in a mean convex ball sharing the same connected (as in, single) boundary component are isotopic to each other and furthermore standard; see also Theorem 3.1 in his paper with Frohman \cite{fm}. The proofs in either go by showing, after making the same reduction to boundary components as above, that nonflat minimal surfaces in a domain with mean convex boundary must intersect by a moving plane argument in the first or as in the proof of the halfspace theorem of Hoffman and Meeks \cite{HM1} in the second. The moving plane argument doesn't apply in our setting; but the Frankel theorem for $f$-Ricci positive metrics, which the Gaussian metric is, could provided the boundary of the ball under consideration is mean convex. See theorem 6.4 in \cite{WW}. There might problems due to possible noncompactness in the lift -- although one might hope this plays out better than in the Frankel theorem below in terms of conditions needed since the Gaussian metric restricted on finite domains has bounded curvature. The main issue with this seems to be that spheres of large (Euclidean) radius, or more general convex sets of large in-radius, are not mean convex in the Gaussian metric: the ball considered might have to in fact be quite large in the proof because it is picked so that the minimal surface under consideration consists of $k$ annuli in its complement (here, $k =1$). And indeed one can see suitably large domains, in particular those strictly containing the round shrinking sphere, will never be shrinker mean convex at all points of their boundary in the Gaussian metric using the flow and maximum principle, ruling out more clever design of domains. A possible way to deal with this might be, upon intersecting with the ball, also to perturb the function $f$ to make the boundary mean convex: even granting this such a perturbation would also require positivity in the Hessian of $f$ and it was not immediately clear this could be guaranteed to the author at the time of this writing, although perhaps this can be arranged by a more careful examination of the perturbation function one would use. 
$\medskip$

Reducing down to surfaces with boundary is not strictly necessary of course, and more in line with Lawson's original argument in \cite{L} one might ask if a Frenkel theorem could be applied directly in the noncompact setting to rule out there being two boundary components $A$ and $B$ of $\widehat{R_{out,in}}$ discussed in the proof of theorem \ref{thm:unknotted}. Wei and Wylie's Frankel property for general $f$-Ricci positive metrics doesn't apply in this case because it requires $f$ be bounded, which it isn't in our setting if we don't consider subdomains with boundary. This leaves to the authors knowledge the following two statements to consider applying, the first due to Impera, Pigola, and Rimoldi and the second very recent one due to Chodosh, Choi, Mantoulidis, and Schulze (the author thanks A. Sun for this reference): 
\begin{thm}[Theorem B in \cite{IPR}] Let $\Sigma_1^m$ and $\Sigma_2^m$ be properly embedded connected self-shrinkers in the Euclidean space $\R^{m+1}$. Assume that $\Sigma_2$ has a uniform regular normal neighborhood $\mathcal{T}(\Sigma_2)$. If 
\begin{equation}
\liminf\limits_{|z| \to \infty, z \in \Sigma_2} \frac{\text{dist}_{\R^{m+1}}(z, \Sigma_1)}{e^{-b|z|^2}\mathcal{P}(|z|)^{-1}} > 0
\end{equation}
for some polynomial $\mathcal{P} \in \R[t]$ and some constant $0 \leq b < \frac{1}{2}$, then $\Sigma_1 \cap \Sigma_2 \neq \emptyset$.
\end{thm}
In the following, $F$-stationary means stationary with respect to Gaussian area: 

\begin{thm}[Corollary C.4 in \cite{CCMS}] If $V, V'$ are $F$-stationary varifolds, then supp $V \cap$ supp $V' \neq \emptyset$.
\end{thm}

In the first statment above properness enters because for self shrinkers in $\R^m$ it guarantees polynomial volume growth by the result of Ding and Xin \cite{DX}. The issue though is that there are cases where the universal cover of a bounded (and hence polynomial volume growth) surface, such as surfaces of genus greater than 2, has exponential volume growth, so it is not obvious that the first statement can be applied in the lift to the boundary components. This is a purely topological phenemonon in fact and hence the volume growth theorem of Ding and Xin couldn't pass to the universal cover. In fact, the boundary components could be stable for the same reason; in general it is known that the spectrum of the Laplacian (and imaginably more general elliptic operators, such as Jacobi operators) may decrease upon lifting to universal cover, unless the fundamental group is amenable: see Brooks \cite{Brooks} -- the author thanks R. Unger for bringing this paper to his attention. We point out here this was a detail overlooked in \cite{MW1} (particularly claim 2.1) which can be fixed in that argument by lifting a perturbation by eigenfunction of the Jacobi operator to the universal cover.
$\medskip$

Now there are cases that stable minimal surfaces can be classified irrespective of volume growth in more general ambient manifolds than $\R^{m}$ which one might hope circumvent this issue, for instance in positive scalar curvature, but to the author's knowledge the best result that applies seems to just be that a stable minimal surface in the lifted Gaussian metric must be locally quadratic in a quantitative sense, see \cite{CM4}. This doesn't imply stronger results certainly can't be true for the lifted Gaussian metric, of course. 
$\medskip$

The proof of the second statement above uses the fact that shrinkers ``collapse'' onto the origin in $\R^{n+1}$. In the case of two smooth self shrinkers where one is compact its a simple consequence of the avoidance principle: the distance between them must not decrease but on the other hand they both shrink to the origin after one ``second'' under the flow The proof of the more general statement is an elaboration on this using Ilmanen's localized avoidance principle. However, this argument doesn't seem to apply when passing to coverings so does not seem to apply to the lifts $A$ and $B$ of $M$ discussed in section 4. Indeed it is true they should never intersect, the issue is that there seems to be no good reason that their flows should approach a common point in contrast to shrinkers in $\R^3$ because the origin could be the preimage of many points in the universal cover. For example if $M$ is a self shrinker where the origin lays within the region $R_{out}$ bounded by $M$ but the lift of $R_{out}$ has two connected components $A$ and $B$, it seems that $A$ and $B$ should retreat from each other, and in so doing not giving a contradiction, because their corresponding (lifts of) origin(s) are ``behind'' $A$ and $B$. 
$\medskip$

Instead of working entirely in the universal cover, the next thing one might try, inspired by the ideas above, could be to find a minimal surface in the universal cover proceeding as in the author's joint paper with S. Wang \cite{MW1} and then project it back down to $\R^3$ and use a Frankel theorem there. In the compact case one can deal with the incompleteness of the Gaussian metric by an intermediate perturbation argument as in \cite{MW1} or Brendle's paper \cite{Bren} to then find a stable minimal surface $\widehat{N}$ in the Gaussian metric in $\widehat{R_{out, in}}$ by solving a sequence of Plateau problems over domains exhausting a connected component  -- note since it might concievably not have polynomial volume growth this itself does not give a contradiction, as mistakenly claimed by the author as an aside in his thesis \cite{Mra1}. The noncompact case is more delicate since $M$ will always have points which lay in the ``perturbed region'' but by chosing the perturbations correctly as Brendle does in proposition 12 of his paper the same construction seems to work in this case as well. At any rate, note that a $\widehat{N}$ found by such means is concievably not equivariant under deck transformations (in the exhaustion of $A$, none of the domains would be equivariant -- although perhaps if one cound pick the exhaustion to be comprised of whole fundamental domains/''tiles'' of $A$ it seems plausible the limit minimal surface obtained might be equivariant) so in fact might be sent by the covering map to something that is at least intrinsically smooth but nonproper. This possible nonproperness seems to give technical issues, at least in the noncompact case where the distance between the two shrinkers could possibly be zero; in the compact case, or where there is positive distance between the two shrinkers, it seems to be fine arguing as in the author's thesis \cite{Mra1} by slightly ``tilting'' the compact one to use the classical avoidance principle -- this is another issue overlooked in \cite{MW1}.  This certainly rules out immediately invoking the result of Impera, Pigola, and Rimoldi, and it also seems to rule out invoking at least without care the Frankel property of Chodosh, Choi, Mantoulidis, and Schulze because there are local area boundedness assumptions in most of the literature on Brakke and level set flows which may affect, for instance, arguments where the Brakke inequality is applied. Of course, once we obtain the self shrinker we do from the flow, then their Frankel theorem can be applied. 
$\medskip$

To avoid these issues of potentional large volume growth, nonproperness, and nonequivariance then, it seems most expedient to work in the base as much as possible when appealing to the geometric facts at hand. The flow argument given above essentially does this, where we go to the universal cover only to check that the flow never becomes nonempty. An alternate path may be to appeal to \cite{MSY}, where Meeks, Simon, and Yau developed a Plateau problem type approach to minimizing in an isotopy class -- specifically see sections 4 and 6 of their paper. The issue of incompleteness of the Gaussian metric still remains though and bounded geometry is assumed in their work but because the incompleteness of the Gaussian metric is ``at infinity,'' one might apply Meeks--Simon--Yau in compact exhaustions of $R_{out, in}$, with the metric perturbed near the boundary of these sets as in Brendle. Since the Gaussian metric has bounded curvature on any compact subset, standard curvature estimates should then allow one to pass to a stable limit: indeed, this is essentially how the noncompact case is handled in section 6 in their paper. The limit, provided its nondegenerate which could concievably happen, should be nonempty by an intersection number argument in the universal cover similar to above, should be stable as a limit of stable minimal surfaces and should have polynomial volume growth by a comparison argument. Of course, merely knowing the limit is disjoint from $M$ and satisfies the properness conditions discussed above suffices to garner a contradiction by the Frankel theorem. It seems perturbations by eigenfunctions wouldn't be necessary in this approach, so the full statement in the case of cylindrical ends might also be attained this way. 
$\medskip$

It might be interesting to note that the renormalized mean curvature flow argument deals with the incompleteness of the Gaussian metric in a way entirely different from perturbing the metric, where by instead of flowing by the mean curvature flow in the Gaussian metric the ``worst'' part of the speed function is discarded. Besides potentially being of independent interest, we also note the flow method is also well suited to studying shrinkers under an entropy assumption in the spirit of \cite{CIMW, BW} because the entropy is monotone.


\begin{thebibliography}{9}



\bibitem{BA} Ben Andrews. \textit{Non-collapsing in mean-convex mean curvature flow}. Geom.Topol. 16, 3 (2012), 1413-1418.

\bibitem{ALM} Ben Andrews, Mat Langford, and James McCoy. \textit{Non-collapsing in fully nonlinear curvature flows}. Ann. Inst. H. Poincaré Anal. Non Linéaire no. 30 (2013), 23–32.



\bibitem{AV} Sigurd Angenent and J.J.L. Velázquez.  \textit{Degenerate neckpinches in mean curvature flow}. J. reine angew. Math. 482 (1997), 15-66. 




\bibitem{BW} Jacob Bernstein and Lu Wang. \textit{A Topological Property of Asymptotically Conical Self-Shrinkers of Small Entropy}. Duke Math. J. 166, no. 3 (2017), 403-435



\bibitem{B} Kenneth Brakke. \textit{The Motion of a Surface by its Mean Curvature}. Princeton University Press, 1978.


\bibitem{Bren} Simon Brendle. \textit{Embedded self-similar shrinkers of genus 0}. Annals of Mathematics, vol. 183, no. 2, 2016, pp. 715–728. 

\bibitem{BH} Simon Brendle and Gerhard Huisken. \textit{Mean curvature flow with surgery of mean convex surfaces in} $\R^3$.  Invent. Math 203, 615-654

\bibitem{Brooks} Robert Brooks. \textit{The fundamental group and the spectrum of the Laplacian}. Comment. Math. Helevetici 56 (1981), 581-598. 
















\bibitem{CGG}  Yun Gang Chen, Yoshikazu Giga, and Shun'ichi Goto. \textit{Uniqueness and existence of viscosity solutions of generalized mean curvature flow equations}. J. Differential Geom. 33 (1991), no. 3, 749--786.


\bibitem{CY} Bing-Long Chen and Le Yin. \textit{Uniqueness and pseudolocality theorems of the mean curvature flow}. Comm. Anal. Geom, Volume 15, Number 3, 435-490, 2007. 



\bibitem{OW} Otis Chodosh. Mean curvature flow (Math 258) lecture notes. Unpublished notes of a class taught by Brian White. 

\bibitem{CCMS} Otis Chodosh, Kyeongsu Choi, Christos Mantoulidis, and Felix Schulze. \textit{Mean curvature flow with generic initial data}. arXiv:2003.14344 




\bibitem{CM} Tobias Colding and William Minicozzi II. \textit{Generic mean curvature flow I; generic singularities}. Annals of Mathematics. (2) 175 (2012), 755-833. 

 \bibitem{CM1} Tobias Colding and William Minicozzi II. \textit{Smooth compactness of self-shrinkers.} Comment. Math. Helv. 87 (2012), 463-475. 
 
\bibitem{CM2} Tobias Colding and William Minicozzi II. \textit{Uniqueness of blowups and Lojasiewicz inequalities}. Ann. of Math. (2) 182 (2015), no. 1, 221–285.

 \bibitem{CM3} Tobias Colding and William Minicozzi II. \textit{The singular set of mean curvature flow with generic singularities.} Inventiones mathematicae 
volume
 204, 
pages
443–471(2016)

\bibitem{CM4} Tobias Colding and William Minicozzi II.  \textit{Estimates for parametric elliptic integrands.} Int. Math. Res. Not. Vol 2002 (6), 291–297

\bibitem{survey} Tobias Colding, William Minicozzi II and Erik Pedersen. \textit{Mean Curvature Flow}. Bull. Amer. Math. Soc. 52 (2015), 297-333

 \bibitem{CIMW} Tobias Colding, Tom Illmanen, William Minicozzi II, and Brian White. \textit{The round sphere minimizes entropy among closed self-shrinkers}. J. Differential Geom. 95 (2013), 53-69

 \bibitem{DX} Qi Ding and Y.L. Xin. \textit{Volume growth eigenvalue and compactness for self-shrinkers.} Asian J. Math.
Volume 17, Number 3 (2013), 443-456.

\bibitem{EH} Klaus Ecker and Gerhard Huisken. \textit{Interior estimates for hypersurfaces moving by mean curvature}. Invent. Math. 105 (1991), 547-569.

\bibitem{EH1} Klaus Ecker and Gerhard Huisken. \textit{Mean curvature evolution of entire graphs}. Ann. Math. 130 (1989), 453-471.


\bibitem{ES} Lawrence Evans and Joel Spruck. \textit{Motion of level sets by mean curvature. I.} J. Differential Geom. 33 (1991), no. 3, 635--681

\bibitem{F} Theodore Frankel. \textit{On the fundamental group of a compact minimal submanifold}. Ann of Math. 83 (1964), 68-73.

\bibitem{Free} Michael Freedman. \textit{An unknotting result for complete minimal surfaces in }$\R^3$. Invent. Math. 109 (1992), 41–46.

\bibitem{fr} Charles Frohman. \textit{The topological uniqueness of triply-periodic minimal surfaces in} $\R^3$. J. Diff. Geom. 31 (1990), 277–283.

\bibitem{fm} Charles Frohman and W. H. Meeks III. \textit{The ordering theorem for the ends of properly embedded minimal surfaces.} Topology 36 (1997), 605–617.

\bibitem{fm2} Charles Frohman and W. H. Meeks III. \textit{The topological uniqueness of complete one-ended minimal surfaces and Heegaard surfaces in} $\R^3$. J. Amer. Math. Soc. 10 (1997), 495–512.

\bibitem{GH} Panagiotis Gianniotis and Robert Haslhofer. \textit{Intrinsic diameter control under the mean curvature flow}. To appear in Amer. J. Math. 




\bibitem{Hall} Peter Hall. \textit{Two topological examples in minimal surface theory}. J. Differential Geom. Volume 19, Number 2 (1984), 475-481.

\bibitem{HK0} Robert Haslhofer and Bruce Kleiner. \textit{Mean curvature flow of mean convex hypersurfaces}. Comm. Pure Appl. Math. 70(3):511--546, 2017.

\bibitem{HK} Robert Haslhofer and Bruce Kleiner. \textit{Mean curvature flow with surgery}. Duke Math. J., Volume 166, Number 9 (2017), 1591-1626.

\bibitem{HKet} Robert Haslhofer and Daniel Ketover. \textit{Minimal 2-spheres in 3-spheres}. Duke Math. J., Volume 168, Number 10 (2019), 1929-1975.







\bibitem{Head} John Head. \textit{The Surgery and Level-Set Approaches to Mean Curvature Flow}. Thesis. 

\bibitem{Head1} John Head. \textit{On the mean curvature evolution of two-convex hypersurfaces}. J. Differential Geom. 94 (2013) 241-266


\bibitem{HW} Or Hershkovits and Brian White. \textit{Sharp entropy bounds for self-shrinkers in mean curvature flow}. Geom. Topol. Volume 23, Number 3 (2019), 1611-1619.

\bibitem{HW1} Or Hershkovits and Brian White. \textit{Avoidance for set-theoretic solutions of mean-curvature-type flows}. Preprint, arXiv:1809.03026

\bibitem{HW2} Or Hershkovits and Brian White. \textit{Nonfattening of Mean Curvature Flow at Singularities of Mean Convex Type}. Comm. Pure Appl. Math. Volume 73, Issue 3 (2020), 558-580.

\bibitem{H} Gerhard Huisken. \textit{Asymptotic behavior for singularities of the mean curvature flow}. J. Differential Geom. 31 (1990), no. 1, 285--299. 






\bibitem{HM} David Hoffman and William Meeks III. \textit{The asymptotic behavior of properly embedded minimal surfaces of finite topology}. J. AMS 2(4) (1989), 667-681. 

\bibitem{HM1} David Hoffman and William Meeks III. \textit{The strong halfspace theorem for minimal surfaces}. Invent. math 101, 373-377 (1990) 

\bibitem{HS} Gerhard Huisken and Carlo Sinestrari. \textit{Mean curvature flow with surgeries of two-convex hypersurfaces}.  Invent. Math 175, 137-221 (2009)

\bibitem{I} Tom Ilmanen. \textit{Singularities of mean curvature flow of surfaces}. preprint, 1995.

\bibitem{I1} Tom Ilmanen. \textit{Elliptic regularization and partial regularity for motion by mean curvature}. Mem. Amer. Math. Soc.108, (1994), no. 520, x+90.

\bibitem{I2} Tom Ilmanen. \textit{Problems in mean curvature flow}. Available at http://people.math.ethz.ch/~ilmanen/classes/eil03/problems03.ps

\bibitem{I3} Tom Ilmanen. \textit{Generalized Flow of Sets by Mean Curvature on a Manifold}. Indiana University Mathematics Journal, vol. 41, no. 3, 1992, pp. 671–705.

\bibitem{IW} Tom Ilmanen and Brian White. \textit{Sharp lower bounds on density of area-minimizing cones}. Camb. J. Math., v.3, 2015, p. 1--18.

\bibitem{INS} Tom Ilmanen, Andr\'e Neves, and Felix Schulze.\textit{ On short time existence for the planar network flow}. J. Differential Geom. Volume 111, Number 1 (2019), 39-89.

\bibitem{IPR} Debora Impera, Stefano Pigola, and Michele Rimoldi. \textit{The frankel property for self-shrinkers from the viewpoint of elliptic PDE's}. Preprint, arXiv:1803.02332

\bibitem{KM} Stephen Kleene and Niels Martin M{\o}ller. \textit{Self-shrinkers with a rotationnal symmetry.} Trans. Amer. Math.Soc., 366(8):3943–3963, 2014.

\bibitem{KKM} Nikolaos Kapouleas, Stephen Kleene, and Niels Martin M{ø}ller. \textit{Mean curvature self-shrinkers of high genus: Non-compact examples}. Journal für die reine und angewandte Mathematik, 2018(739), 1-39.






\bibitem{Lau} Joseph Laurer. \textit{Convergence of mean curvature flows with surgery}. Comm. Anal. Geom., Volume 21, Number 2, 355-363, 2013.















\bibitem{L} Blaine Lawson. \textit{The unknottedness of minimal embeddings}. Invent. Math. 11 (1970), 41–46.


\bibitem{Lin} Longzhi Lin. \textit{Mean curvature flow of star-shaped hypersurfaces}. To appear in Comm Anal Geom.

\bibitem{LS} Zhengjiang Lin and Ao Sun. \textit{Bifurcation of perturbations of non-generic closed self-shrinkers}. Preprint, arXiv:2004.07787

\bibitem{Meeks} William Meeks III. \textit{The topological uniqueness of minimal surfaces in three dimensional Euclidean space}. Topology 20:389-410, 1981. 

\bibitem{MY} William Meeks III and Shing-Tung Yau. \textit{The topological uniqueness of complete minimal surfaces of finite topological type}. Topology 31:305-315, 1992. 

\bibitem{MSY} William Meeks III, Leon Simon, and Shing-Tung Yau. \textit{Embedded minimal surfaces, exotic spheres, and manifolds with positive Ricci curvature.} Annals of Mathematics (1982) 621-659.

\bibitem{MS} Jan Metzger and Felix Schulze. \textit{No mass drop for mean curvature flow of mean convex hypersurfaces}. Duke Math. J. Volume 142, Number 2 (2008), 283-312.

\bibitem{Mra} Alexander Mramor. \textit{Regularity and stability results for the level set flow via the mean curvature flow with surgery}. To appear in Comm Anal Geom. 

\bibitem{Mra1} Alexander Mramor. \textit{On Singularities and Weak Solutions of Mean Curvature Flow}. Thesis. 

\bibitem{MW1} Alexander Mramor and Shengwen Wang. \textit{On the topological rigidity of compact self shrinkers in} $\R^3$. Int. Mat. Res. Not., rny050. 

\bibitem{MW2} Alexander Mramor and Shengwen Wang. \textit{Low entropy and the mean curvature flow with surgery}. To appear in CVPDE. 

\bibitem{Smo} Kurt Smoczyk. \textit{Starshaped hypersurfaces and the mean curvature flow}. Manuscripta Math. 95 (1998), no. 2, 225–236.

\bibitem{SW} Weimin Sheng and Xu-Jia Wang. \textit{Singularity Profile in the Mean Curvature Flow}. Methods Appl. Anal. Volume 16, Number 2 (2009), 139-156.


\bibitem{W}  Friedhelm Waldhausen. \textit{Heegaard-zerlegungen der 3-sphere}. Topology 7:195-203, 1968.

\bibitem{Lu} Lu Wang. \textit{Asymptotic structure of self-shrinkers}. Preprint, arxiv: 1610.04904

\bibitem{Lu1} Lu Wang. \textit{Uniqueness of self-similar shrinkers with asymptotically cylindrical ends}. J. Reine Angew. Math. 715 (2016), 207-230.


\bibitem{WW} Guofang Wei, Will Wylie. \textit{Comparison geometry for the Bakry-Emery Ricci tensor}. J. Differential Geom. Volume 83, Number 2 (2009), 337-405.


\bibitem{W2} Brian White. \textit{The topology of hypersurfaces moving by mean curvature}. Comm. Anal. Geom. (2) 3 (1995) 317-333. 

\bibitem{W3} Brian White. \textit{Stratification of minimal surfaces, mean curvature flows, and harmonic maps.}  J. reine angew. Math. 488 (1997), 1-35. 

\bibitem{W} Brian White. \textit{The size of the singular set in mean curvature flow of mean-convex sets}.  J. Amer. Math. Soc. 13 (2000), 665-695. 

\bibitem{W1} Brian White. \textit{A local regularity theorem for mean curvature flow}.  Annals of Mathematics. (2) 161 (2005),  1487-1519. 




\end{thebibliography}
\end{document}